\documentclass[preprint]{elsarticle}
\usepackage{amssymb,amsmath,amsthm}
\usepackage{graphicx}
\usepackage{xcolor}
\usepackage{hyperref}
\usepackage[numbers]{natbib}

\newcommand{\BW}{\mathbb{W}}
\newcommand{\complex}{\mathbb{C}}
\newcommand{\eps}{\epsilon}

\newcommand{\cB}{\mathcal{B}}
\newcommand{\cM}{\mathcal{M}}
\newcommand{\cI}{\mathcal{I}}

\newcommand{\bv}{\mathbf{v}}
\newcommand{\bu}{\mathbf{u}}
\newcommand{\bbf}{\mathbf{f}}

\newcommand{\dsp}{\displaystyle}
\newcommand{\dom}{\text{dom}\,}

\newtheorem{theorem}{Theorem}
\newtheorem{remark}{Remark}
\newtheorem{definition}{Definition}
\newtheorem{corollary}{Corollary}
\newtheorem{proposition}{Proposition}
\newtheorem{lemma}{Lemma}
\newtheorem{example}{Example}
\newcommand{\M}[1]{\left({#1}\right)}
\newcommand{\norm}[1]{\left\|{#1}\right\|}
\DeclareMathOperator{\sgn}{sgn}
\DeclareMathOperator{\id}{id}
\DeclareMathOperator{\adj}{adj}
\DeclareMathOperator{\interior}{int}

\begin{document}

\title{Pseudospectra of Isospectrally Reduced\\ Matrices and Systems}

\author[uu]{Fernando Guevara Vasquez}
\ead{fguevara@math.utah.edu}

\author[byu]{Benjamin Z. Webb\corref{cor1}\fnref{tru}}
\ead{bwebb@math.byu.edu}

\address[uu]{Mathematics Dept., JWB 233, University of Utah, Salt Lake City, UT 84112, USA}
\address[byu]{Dept. of Mathematics, 308 TMCB, Brigham Young University, Provo, UT 84602, USA}
\fntext[tru]{{\em Present address:} Laboratory of Statistical Physics, The Rockefeller University, 1230 York Avenue, New York, NY 10065. Phone: XXX-XXX-XXXX. Fax: XXX-XXX-XXXX.}

\cortext[cor1]{Corresponding author}

\begin{abstract}
The isospectral reduction of matrix, which is closely related to its Schur complement, allows to reduce the size of a matrix while maintaining its eigenvalues up to a known set. Here we generalize this procedure by increasing the number of possible ways a matrix can be isospectrally reduced. The reduced matrix has rational functions as entries. We show that the notion of pseudospectrum can be extended to this class of matrices and that the pseudospectrum of a matrix shrinks as the matrix is reduced. Hence the eigenvalues of a reduced matrix are more robust to entry-wise perturbations than the eigenvalues of the original matrix. We also introduce the notion of inverse pseudospectrum (or pseudoresonances), which indicates how stable the poles of a matrix with rational function entries are to certain matrix perturbations. A mass spring system is used to illustrate and give a physical interpretation to both pseudospectra and inverse pseudospectra.
\end{abstract}


\maketitle

\section{Introduction}
The process of isospectrally reducing a matrix was first considered in \cite{Bunimovich:2012:IGT}, where it was shown that a weighted digraph could be reduced while maintaining the eigenvalues of the graph's weighted adjacency matrix, up to a known set. The motivation in this setting was to allow one to simplify the structure of a complicated network (graph) while preserving its spectral information. One of the main results of this paper is that any weighted digraph can be uniquely reduced to a graph on any subset of its nodes via some sequence of isospectral reductions. 

Later it was shown in \cite{Bunimovich:2011:IGR} that such matrix reductions could be used to improve the classical eigenvalue estimates of Gershgorin, Brauer, Brualdi, and Varga \cite{Gershgorin:1931:UDA,Brauer:1947:LCR,Brualdi:1982:MED,Varga:2004:GC}. Specifically, the eigenvalue estimates associated with Gershgorin and Brauer both improve for any matrix reduction and can be successively improved by further matrix reductions. The eigenvalue estimates of Brualdi and Varga are more complicated but can be shown to improve for specific types of matrix reductions.

In this paper we generalize this previous work by first showing that a matrix can be isospectrally reduced over any of its principal submatrices, under mild conditions. This is an improvement over the isospectral reduction method presented in \cite{Bunimovich:2011:IGR,Bunimovich:2012:IGT,Bunimovich:2012:IC}, since in these three papers a submatrix is required to have a particular form in order for the reduction to be defined. This fundamental improvement allows us to avoid the sequence of reductions that were previously necessary for certain matrix reductions. We prove in a more general setting that a sequence of reductions still leads to a uniquely defined matrix that depends only on the final reduction (see theorem \ref{theorem2}).

As defined in \cite{Bunimovich:2011:IGR} a matrix with rational function entries has both a spectrum and an inverse spectrum. When a square matrix is isospectrally reduced, the result is a smaller matrix that again has a spectrum and an inverse spectrum. The relation between the spectrum and inverse spectrum of the reduced and unreduced matrices is dictated by the specific submatrix over which the matrix is reduced (see theorem \ref{maintheorem}).

Expanding on the work done in \cite{Bunimovich:2011:IGR}, we show that it is possible to not only use the eigenvalue estimates associated with Gershgorin to estimate the eigenvalues of a matrix, but to estimate its inverse eigenvalues. This is done by introducing the concept of the spectral inverse of a matrix, i.e. the matrix in which the eigenvalues are the inverse eigenvalues of the original matrix and vice versa. Therefore, the results found in \cite{Bunimovich:2011:IGR} allow us to give estimates of the inverse eigenvalues of a matrix and use matrix reductions to improve them (see theorem \ref{theorem4}).

Another reason we care about isospectral reductions is that they naturally arise in network models when we do not have access to all the network nodes. We use a mass spring network to illustrate this: the isospectral reduction amounts to the response of the network where we only have access to some terminal nodes (see example \ref{ex:spring0}). In the case where all nodes are accessible (i.e. all nodes are terminal nodes), the eigenvalues correspond to frequencies for which there is a non-zero node displacement that results in zero forces. For the reduced matrix, the eigenvalues indicate frequencies for which a non-zero displacement of the terminal nodes generates zero forces at the terminals. The inverse eigenvalues of the reduced matrix correspond to resonance frequencies, i.e. frequencies for which there is an extremely large force generated by a finite displacement of the terminals.

The pseudospectrum of a matrix gives us the scalars that behave like eigenvalues to within a certain tolerance. This concept is particularly useful in analyzing the properties of matrices that are non-normal, i.e. do not have an orthogonal eigenbasis. The pseudospectrum of a complex valued matrix has been introduced independently many times (see \cite{Trefethen:2005:SP} for details). It has also been studied in the case of matrix polynomials \cite{Lancaster:2005:PMP,Boulton:2008:PMP}. Here, we extend the definition of pseudospectrum to matrices with rational function entries. By use of the spectral inverse we also define the inverse pseudospectrum of a matrix in Section \ref{sec:psire}.

As with complex valued matrices, the pseudospectra we define for a matrix with rational function entries is a subset of the complex plane whose elements behave, within some tolerance, as eigenvalues. Similarly, the inverse pseudospectra of a matrix is the set of scalars that act as inverse eigenvalues for a given tolerance. For the mass spring network we consider, the pseudospectra of the stiffness matrix are the values for which there are node displacements that generate forces that are {\em small} relative to the displacement. The same is true of inverse pseudospectra, except these give way to forces that are {\em large} relative to the displacement. A given tolerance determines how ``large'' and ``small'' these forces are.

We show that the pseudospectra of a reduced matrix are always contained in the pseudospectra of the original matrix for a given tolerance. This implies that the eigenvalues of a reduced matrix are less susceptible to perturbations than the original one. 

The paper is organized as follows. Section \ref{sec:matred} introduces and extends the theory of isospectral matrix reductions. This section also includes the spectral inverse of a matrix along with the Gershgorin type estimates of a matrix' inverse eigenvalues. In Section \ref{sec:psi} we define the pseudospectrum and inverse pseudospectrum of a matrix with rational function entries and show that the pseudospectrum of a matrix shrinks in size as the matrix is reduced. Throughout the paper we consider numerous examples, including the mass spring network mentioned above, which is used to give a physical interpretation to the concepts introduced in this paper.

\section{Isospectral Matrix Reductions}\label{sec:matred}
In the first part of this paper we introduce the class of matrices we wish to consider; namely those matrices which have rational function entries. The reason we consider this class of matrices, as mentioned in the introduction, is that such matrices arise naturally if we wish or need to reduce the size of a matrix (or system) we are considering while maintaining its spectral properties. This procedure of isospectrally reducing a matrix and describing the spectrum of such matrices is the main focus of this section.

\subsection{Matrices with Rational Function Entries}\label{sec:mat}
The class of matrices we consider are those square matrices whose entries are rational functions of $\lambda$. Specifically, let $\complex[\lambda]$ be the set of polynomials in the complex variable $\lambda$ with complex coefficients. We denote by $\BW$ the set of rational functions of the form $$w(\lambda)=p(\lambda)/q(\lambda)$$
where $p(\lambda),q(\lambda)\in\complex[\lambda]$ are polynomials having no common linear factors and where $q(\lambda)$ is not identically zero.

More generally, each rational function $w(\lambda)\in\mathbb{W}$ is expressible in the form
$$w(\lambda)=\frac{a_i\lambda^i+a_{i-1}\lambda^{i-1}+\dots+a_0}{b_j\lambda^j+b_{j-1}\lambda^{j-1}+\dots+b_0}$$ where, without loss in generality, we can take $b_j=1$. The domain of $w(\lambda)$ consists of all but a finite number of complex numbers for which $q(\lambda)=b_j\lambda^j+b_{j-1}\lambda^{j-1}+\dots+b_0$ is zero.

Addition and multiplication on the set $\mathbb{W}$ are defined as follows. For $p(\lambda)/q(\lambda)$ and $r(\lambda)/s(\lambda)$ in $\mathbb{W}$ let
\begin{align}
\label{eq:add}\left(\frac{p}{q}+\frac{r}{s}\right)(\lambda)=& \frac{p(\lambda)s(\lambda)+q(\lambda)r(\lambda)}{q(\lambda)s(\lambda)}; \ \text{and}\\
\label{eq:mult}\left(\frac{p}{q}\cdot\frac{r}{s}\right)(\lambda)=&\frac{p(\lambda)r(\lambda)}{q(\lambda)s(\lambda)}
\end{align}
where the common linear factors in the right hand side of equations (\ref{eq:add}) and (\ref{eq:mult}) are canceled. The set $\mathbb{W}$ is then a field under addition and multiplication.

Because we are primarily concerned with the eigenvalues of a matrix, which is a set that includes multiplicities, we note the following. The element $\alpha$ of the set $A$ that includes multiplicities has \textit{multiplicity} $m$ if there are $m$ elements of $A$ equal to $\alpha$. If $\alpha\in A$ with multiplicity $m$ and $\alpha\in B$ with multiplicity $n$ then\\
\indent (i) the \textit{union} $A\cup B$ is the set in which $\alpha$ has multiplicity $m+n$; and\\
\indent (ii) the \textit{difference} $A-B$ is the set in which $\alpha$ has multiplicity $m-n$ if $m-n>0$ and where $\alpha\notin A-B$ otherwise.

\begin{definition}\label{def1.1}
Let $\BW^{n\times n}$ denote the set of $n\times n$ matrices with entries in $\BW$. For a matrix $M(\lambda)\in\BW^{n\times n}$ the determinant $$\det\big(M(\lambda)-\lambda I\big) = p(\lambda)/q(\lambda)$$
for some $p(\lambda)/q(\lambda)\in\mathbb{W}$. The {\em spectrum} (or eigenvalues) of $M(\lambda)$ is defined as
\[
\sigma\M{M} = \{\lambda\in\complex:p(\lambda)=0\}.
\]
The {\em inverse spectrum} (or \emph{resonances}) of $M(\lambda)$ is defined as
\[
 \sigma^{-1}\M{M} = \{\lambda\in\complex:q(\lambda)=0\}.
\]
\end{definition}

Both $\sigma(M)$ and $\sigma^{-1}(M)$ are understood to be sets that include multiplicities. For example, if the polynomial $p(\lambda)\in\mathbb{C}[\lambda]$ factors as
$$p(\lambda)=\prod_{i=1}^m(\lambda-\alpha_i)^{n_i} \ \ \text{for} \ \alpha_i\in\mathbb{C} \ \text{and} \ n_i\in\mathbb{N}$$
then $\{\lambda\in\mathbb{C}:p(\lambda)=0\}$ is the set in which $\alpha_i$ has multiplicity $n_i$.

\begin{remark}
Since $\mathbb{C}\subset\mathbb{W}$, definition \ref{def1.1} is an extension of the standard definition of the eigenvalues of a matrix to the larger class of matrices $\mathbb{W}^{n\times n}$. In particular, if $M\in\mathbb{C}^{n\times n}$ then $\sigma(M)$ are the standard eigenvalues of $M$.
\end{remark}

In what follows we may, for convenience, suppress the dependence of the matrix $M(\lambda)\in\mathbb{W}^{n\times n}$ on $\lambda$ and simply write $M$. One reason for this is that for much of what we do in this paper we do not evaluate $M(\lambda)$ at any particular point $\lambda\in\mathbb{C}$. Rather, we consider $M$ formally as a matrix with rational function entries.

However, when we do consider the matrix $M(\lambda)\in\mathbb{W}^{n\times n}$ to be a function of $\lambda$ we mean $M$ is the function
$$M:\dom(M)\rightarrow\mathbb{C}^{n\times n},$$
where $\dom(M)$ are the complex numbers $\lambda$ for which every entry of $M(\lambda)$ is defined. Surprisingly, it may be the case that $\sigma(M)\nsubseteq \dom(M)$ as the following example shows.

\begin{example}\label{ex:0}
Consider the matrix $M(\lambda)\in\mathbb{W}^{2\times 2}$ given by
$$M(\lambda)=\left[\begin{array}{cc}
0&\frac{1}{\lambda}\\
0&0
\end{array}\right].$$
As one can compute, $\det(M(\lambda)-\lambda I)=\lambda^2$ implying $\sigma(M)=\{0,0\}$. Therefore, $\sigma(M)\nsubseteq \dom(M)$.
\end{example}

\subsection{Isospectral Matrix Reductions}\label{sec:imr}
We can now describe an {\em isospectral reduction} of a matrix $M\in\mathbb{W}^{n\times n}$. We then compare the spectrum of $M$ to the spectrum of its isospectral reduction.

Let $M\in\mathbb{W}^{n\times n}$ and $N=\{1,\ldots,n\}$. If the sets $\mathcal{R},\mathcal{C}\subseteq N$ are non-empty we denote by $M_{\mathcal{R}\mathcal{C}}$ the $|\mathcal{R}| \times |\mathcal{C}|$ \emph{submatrix} of $M$ with rows indexed by $\mathcal{R}$ and columns by $\mathcal{C}$. Suppose the non-empty sets $\mathcal{B}$ and $\mathcal{I}$ form a partition of $N$. The {\em Schur complement} of $M_{\mathcal{II}}$ in $M$ is the matrix
\begin{equation}
M/M_{\mathcal{II}} = M_{\mathcal{BB}} - M_{\mathcal{BI}} M_{\mathcal{II}}^{-1} M_{\mathcal{IB}},
 \label{eq:schur}
\end{equation}
assuming $M_{\mathcal{II}}$ is invertible.

The Schur complement arises in many applications. For example, if the matrix $M$ is the Kirchhoff matrix of a network of resistors with $n$ nodes then its Schur complement is the Dirichlet to Neumann (or voltage to currents) map of the network given by considering the nodes in $\mathcal{B}$ as terminal or \emph{boundary} nodes and the nodes in $\mathcal{I}$ as \emph{interior} nodes (see e.g. \cite{Curtis:1998:cpg}). A physical interpretation of an isospectral reduction is given in example~\ref{ex:spring0}.

We are now ready to define the isospectral reduction of a matrix $M\in\BW^{n\times n}$.

\begin{definition}\label{def1}
For $M(\lambda) \in\BW^{n\times n}$ let $\mathcal{B}$ and $\mathcal{I}$ form a non-empty partition of $N$. The {\em isospectral reduction} of $M$ over the set $\mathcal{B}$ is the matrix
\begin{equation}
 R_{\lambda}(M;\mathcal{B})=M_{\mathcal{BB}} - M_{\mathcal{BI}}(M_{\mathcal{II}}-\lambda I)^{-1} M_{\mathcal{IB}} \in \BW^{|\mathcal{B}|\times|\mathcal{B}|}.
 \label{eq:isred}
\end{equation}
if the matrix $M_{\mathcal{II}}-\lambda I$ is invertible.
\end{definition}

Note that the reduced matrix $R_{\lambda}(M;\mathcal{B})$ is a Schur complement plus a multiple of the identity:
\begin{equation}\label{eq:sch}
R_{\lambda}(M;\mathcal{B})=(M-\lambda I)/(M_{\mathcal{II}}-\lambda I)+\lambda I.
\end{equation}
More often than not we suppress the dependence of $R_{\lambda}(M;\mathcal{B})$ on $\lambda$ and instead write it as $R(M;\mathcal{B})$.

\begin{example}\label{ex:2}
Consider the matrix $M\in\BW^{6\times 6}$ with $(0,1)$-entries given by
$$M=\left[
\begin{array}{cccccc}
0&0&1&1&0&0\\
0&1&0&0&1&1\\
1&0&1&0&0&0\\
0&1&0&1&0&0\\
1&0&0&0&0&0\\
0&1&0&0&0&0
\end{array}
\right].$$
For $\mathcal{B}=\{1,2\}$ and $\mathcal{I}=\{3,4,5,6\}$ we have
$$(M_{\cI\cI}-\lambda I)^{-1}=\left[
\begin{array}{cccc}
\frac{1}{1-\lambda}&0&0&0\\
0&\frac{1}{1-\lambda}&0&0\\
0&0&-\frac{1}{\lambda}&0\\
0&0&0&-\frac{1}{\lambda}
\end{array}
\right].$$
The isospectral reduction of $M$ over $\mathcal{B}=\{1,2\}$ is then defined as
\[
\begin{aligned}
R(M;\mathcal{B})&=
\left[
\begin{array}{cc}
0&0\\
0&1
\end{array}
\right]-
\left[
\begin{array}{cccc}
1&1&0&0\\
0&0&1&1
\end{array}
\right]
\left[
\begin{array}{cccc}
\frac{1}{1-\lambda}&0&0&0\\
0&\frac{1}{1-\lambda}&0&0\\
0&0&-\frac{1}{\lambda}&0\\
0&0&0&-\frac{1}{\lambda}
\end{array}
\right]
\left[
\begin{array}{cc}
1&0\\
0&1\\
1&0\\
0&1
\end{array}
\right].\\
&=\left[
\begin{array}{cc}
\frac{1}{\lambda-1}&\frac{1}{\lambda-1}\\
\frac{1}{\lambda}&\frac{\lambda+1}{\lambda}
\end{array}
\right]\in\BW^{2\times 2}.
\end{aligned}
\]
\end{example}

If a matrix has an isospectral reduction the spectrum and inverse spectrum of the isospectral reduction and the original matrix are related as follows.

\begin{theorem}\label{maintheorem}\textbf{(Spectrum and Inverse Spectrum of Isospectral Reductions)}\\
For $M(\lambda) \in\BW^{n\times n}$ let $\mathcal{B}$ and $\mathcal{I}$ form a non-empty partition of $N$. If $R_{\lambda}(M;\mathcal{B})$ exists then its spectrum and inverse spectrum are given by
\begin{align*}
\sigma\big(R(M;\mathcal{B})\big)&= \M{\sigma(M)\cup\sigma^{-1}(M_{\mathcal{II}})}-
            \M{\sigma (M_{\mathcal{II}})\cup\sigma^{-1}(M)};
	    ~\text{and}\\
\sigma^{-1}\big(R(M;\mathcal{B})\big)&=\M{\sigma(M_{\mathcal{II}})\cup\sigma^{-1}(M)}-\M{\sigma(M)\cup\sigma^{-1} (M_{\mathcal{II}})}.
\end{align*}
\end{theorem}

\begin{proof}
For $M\in\BW^{n\times n}$, we may assume without loss of generality that $M$ has the block matrix form
\begin{equation}\label{eq3.1}
M=\begin{bmatrix}
M_{\mathcal{II}} & M_{\mathcal{IB}}\\
M_{\mathcal{BI}} & M_{\mathcal{BB}}
\end{bmatrix}
\end{equation}
where $M_{\mathcal{II}}-\lambda I$ is invertible.

Note that the determinant of a matrix and that of its Schur complement are related by the identity
\begin{equation}\label{eq:detschur}
 \det\begin{bmatrix}
 A & B\\
 C & D
 \end{bmatrix}
 =\det(A)\cdot\det(D-CA^{-1}B),
\end{equation}
provided the submatrix $A$ is invertible. Using this identity on the matrix $M-\lambda I$ yields
\[
 \det(M-\lambda I)=\det(M_{\mathcal{II}}-\lambda I)
              \cdot\det\M{(M_{\mathcal{BB}}-\lambda I)-M_{\mathcal{BI}} (M_{\mathcal{II}}-\lambda I)^{-1}M_{\mathcal{IB}}}.
\]
Therefore,
\begin{equation*}
 \det\big(R(M;\mathcal{B})-\lambda I\big)=\frac{\det(M-\lambda I)}{\det(M_{\mathcal{II}}-\lambda I)}.
\end{equation*}
To compare the eigenvalues of $R(M;\mathcal{B})$, $M$, and $M_{\mathcal{II}}$ write
\[
\det(M-\lambda I)=\frac{p(\lambda)}{q(\lambda)}\ \ \text{and} \ \ \det(M_{\cI\cI}-\lambda I)=\frac{t(\lambda)}{u(\lambda)},
\]
for some $p/q,t/u\in\mathbb{W}$. Hence,
\[
 \det\big(R(M;\mathcal{B})-\lambda I\big)=\frac{p(\lambda)u(\lambda)}{q(\lambda)t(\lambda)}.
\]
Let $P=\{\lambda\in\complex:p(\lambda)=0\}$, $Q=\{\lambda\in\complex:q(\lambda)=0\}$, $T=\{\lambda\in\complex:t(\lambda)=0\}$, and $U=\{\lambda\in\complex:u(\lambda)=0\}$, with multiplicities. By canceling common linear factors, definition \ref{def1.1} implies
\begin{align*}
\sigma\big(R(M;\mathcal{B})\big)=&\{\lambda\in\complex:p(\lambda)u(\lambda)=0\}-\{\lambda\in\complex: q(\lambda)t(\lambda)=0\}\\
=&(P\cup U)-(Q\cup T); \ \text{and}\\
\sigma^{-1}\big(R(M;\mathcal{B})\big)=&\{\lambda\in\complex: q(\lambda)t(\lambda)=0\}-\{\lambda\in\complex:p(\lambda)u(\lambda)=0\}\\
=&(Q\cup T)-(P\cup U).
\end{align*}
Since $P=\sigma(M)$, $Q=\sigma^{-1}(M)$, $T=\sigma^{-1}(M_{\mathcal{II}})$, and $R=\sigma(M_{\mathcal{II}})$ the result follows.
\end{proof}

Since a matrix $M\in\mathbb{C}^{n\times n}$ has no inverse spectrum (i.e. $\sigma^{-1}(M) = \emptyset$), theorem~\ref{maintheorem} applied to complex valued matrices has the following corollary.

\begin{corollary}\label{cor1}
For $M(\lambda) \in\complex^{n\times n}$ let $\mathcal{B}$ and $\mathcal{I}$ form a non-empty partition of $N$. Then
\[
\sigma\big(R(M;\mathcal{B})\big)=\sigma(M)-\sigma(M_{\cI\cI})
\ \ \text{and} \ \
\sigma^{-1}\big(R(M;\mathcal{B})\big) =\sigma(M_{\cI\cI})-\sigma(M).
\]
\end{corollary}

\begin{example}\label{ex:3}
Let $M$, $\mathcal{B}$ and $\mathcal{I}$ be as in example~\ref{ex:2}. As one can compute $\sigma(M)=\{2,-1,1,1,0,0\}$ and $\sigma(M_{\mathcal{II}})=\{1,1,0,0\}$. By corollary~\ref{cor1} we then have
\[
\begin{aligned}
\sigma\big(R(M;\mathcal{B})\big)& = \{2,-1,1,1,0,0\}-\{1,1,0,0\} = \{2,-1\}; \ and\\
\sigma^{-1}\big(R(M;\mathcal{B})\big)& = \{1,1,0,0\}-\{2,-1,1,1,0,0\} = \emptyset.
\end{aligned}
\]
Observe that, by reducing $M$ over $\mathcal{B}$ we lose the eigenvalues corresponding to the ``interior'' degrees of freedom $\sigma(M_{\mathcal{II}})=\{1,1,0,0\}$. That is, if we knew both $\sigma(M_{\mathcal{II}})$ and  $\sigma(R(M;\mathcal{B}))$ but not $\sigma(M)$, then corollary~\ref{cor1} states that the set $\sigma(M_{\mathcal{II}})$ is the most by which $\sigma(R(M;\mathcal{B}))$ and $\sigma(M)$ can differ.
\end{example}

Theorem \ref{maintheorem} therefore describes exactly which eigenvalues we may gain from an isospectral reduction and which we may lose. In this way an isospectral reduction of a matrix preserves the spectral information of the original matrix. However, it may not always be possible to find an isospectral reduction of a matrix $M\in\mathbb{W}^{n\times n}$.

For instance, consider the matrix $M\in\mathbb{W}^{2\times 2}$ given by
\begin{equation}\label{mat1}
M=\left[\begin{array}{cc}
1&1\\
1&\lambda
\end{array}\right].
\end{equation}
For $\mathcal{B}=\{1\}$ and $\mathcal{I}=\{2\}$ note that $M_{\mathcal{II}}-\lambda I=[0]$, which is not invertible. Therefore, $M$ cannot be isospectrally reduced over $\mathcal{B}$.

In general there is no way to know beforehand if the isospectral reduction $R(M;\mathcal{B})$ exists without computing $(M_{\mathcal{II}}-\lambda I)^{-1}$. However, the following subset of $\mathbb{W}^{n\times n}$ can always be reduced over any nonempty subset $\mathcal{B}\subset N$.

For any polynomial $p(\lambda)\in\complex[\lambda]$, let $\deg(p)$ denote its degree. If $w(\lambda)=p(\lambda)/q(\lambda)$ where both $p(\lambda),q(\lambda)\in\complex[\lambda]$ are not identically zero we define the degree of the rational function $w(\lambda)$ by $$\pi(w)=\deg(p)-\deg(q).$$ In the case where $p(\lambda)=0$ we let $\pi(w)=0$.

\begin{definition}
We denote by $\BW_{\pi}$ the set of rational functions
$$\BW_{\pi}=\{w\in\BW:\pi(w)\leq0\}$$
and let $\BW^{n\times n}_{\pi}$ be the set of $n\times n$ matrices with entries in $\BW_{\pi}$.
\end{definition}

\begin{lemma}\label{lemma1}
If $M(\lambda) \in\BW^{n\times n}_{\pi}$ and $\mathcal{B}\subset N$ is non-empty then $R_{\lambda}(M;\mathcal{B})\in\mathbb{W}_{\pi}^{|\mathcal{B}|\times |\mathcal{B}|}$.
\end{lemma}

\begin{proof}
Let $M\in\BW^{n\times n}_{\pi}$. The inverse of the matrix $M-\lambda I$ is given by
\begin{equation}\label{eq1}
(M-\lambda I)^{-1}=\frac{1}{\det(M-\lambda I)}\adj(M-\lambda I)
\end{equation}
where $\adj(M-\lambda I)$  is the adjugate matrix of $M-\lambda I$, i.e. the matrix with entries
\begin{equation}
 \adj(M-\lambda I)_{ij} = (-1)^{i+j} \det(\cM_{ji}), ~ 1 \leq i,j \leq n,
\end{equation}
where $\cM_{ij} \in \BW^{(n-1)\times(n-1)}$ is obtained by deleting the $i-$th row and $j-$th column of $M-\lambda I$.

Note that
\begin{equation}
 \det(M-\lambda I)=\sum_{\rho\in \mathcal{P}_n}\Big(\sgn(\rho)\prod_{i=1}^n(M-\lambda I)_{i,\rho(i)}\Big)
 \label{eq:detsum}
\end{equation}
where the sum is taken over the set $\mathcal{P}_n$ of permutations on $N$. The sign $\sgn(\rho)$ of the permutation $\rho\in\mathcal{P}_n$ is  1 (resp. $-1$) if $\rho$ is the composition of an even (resp. odd) number of permutations of two elements.

Using equations (\ref{eq:prod}) and (\ref{eq:lambda}) in \ref{app:degree}, the term in \eqref{eq:detsum} corresponding to the identity permutation $\rho=\id\in\mathcal{P}_n$ has degree $n$ while for $\rho\neq\id$ the other terms have degree strictly smaller than $n$. Equation (\ref{eq:sum}) then implies
\begin{equation}\label{eq:deg1}
\pi\big(\det(M-\lambda I)\big)=n.
\end{equation}
Therefore $\det(M-\lambda I)$ is not identically zero, implying via equation \eqref{eq1} that the inverse $(M-\lambda I)^{-1}$ exists.
Similarly, for $i\in N$ the matrix $\mathcal{M}_{ii}$ is equal to $\widetilde{\mathcal{M}}_{ii}-\lambda I$ for some $\widetilde{\mathcal{M}}\in\mathbb{W}_{\pi}^{(n-1)\times(n-1)}$. Hence,
\begin{equation}\label{eq:deg2}
 \pi\big(\det(\cM_{ii})\big) =n-1, ~ \text{for} ~ i\in N.
\end{equation}

For $i\neq j$ the matrices $\cM_{ij}\in\mathbb{W}^{(n-1)\times(n-1)}$ contain $n-2$ entries of the form $M_{k \ell} - \lambda$ where all other entries of $\cM_{ij}$ belong to the set $\mathbb{W}_{\pi}$. Hence, equations (\ref{eq:prod}) and (\ref{eq:lambda}) imply that for $i\neq j$
\begin{equation}\label{eq:deg3}
\pi\big(\det(\cM_{ij})\big)\leq n-2, ~ \text{for} ~ i,j\in N
\end{equation}
since for $\rho\in\mathcal{P}_{n-1}$ at most $n-2$ terms in the product $\prod_{k=1}^{n-1}(\mathcal{M}_{ij})_{k,\rho(k)}$ have the form $M_{k \ell}-\lambda$.

Given that the degree of $\det(\cM_{ij})$ in (\ref{eq:deg3}) may be zero, equations \eqref{eq:deg1}--\eqref{eq:deg3} together with (\ref{eq:ratio}) imply that $\pi( (M-\lambda I)^{-1}_{ij}) \leq 0$ for all $1\leq i,j\leq n$. Hence, $(M-\lambda I)^{-1} \in \BW_\pi^{n \times n}$. Therefore, if $\mathcal{B}$ and $\mathcal{I}$ form a nonempty partition of $N$ then
$$\left[(M-\lambda I)^{-1}\right]_{\mathcal{II}}\in\mathbb{W}^{|\mathcal{I}|\times|\mathcal{I}|}_{\pi}.$$
Definition \ref{def1} along with equations (\ref{eq:prod}) and (\ref{eq:lambda}) then imply that $R(M;\mathcal{B})$ has entries in $\mathbb{W}_{\pi}$.
\end{proof}

Note that lemma \ref{lemma1} implies the existence of any isospectral reduction $R(M;\mathcal{B})$ if $M\in\mathbb{W}^{n\times n}_{\pi}$ and $\mathcal{B}\subset N$. In particular, any complex valued matrix can be reduced over any index set. Since the matrix $M$ given in (\ref{mat1}) does not belong to  $\mathbb{W}_{\pi}^{2\times 2}$ lemma \ref{lemma1} does not apply in this particular case.

\begin{remark}
Because a matrix $M\in\BW^{n\times n}_{\pi}$ can be reduced over any nonempty index set $\mathcal{B} \subset N$, the isospectral reductions presented here are more general than those given in \cite{Bunimovich:2011:IGR,Bunimovich:2012:IGT,Bunimovich:2012:IC}. In these three papers, for $M$ to be reduced over the index set $\mathcal{B}$ the matrix $M_{\mathcal{II}}$ was required to be similar to an upper triangular matrix. Here, we have no such restriction.
\end{remark}

In the following example we demonstrate how one can use an isospectral reduction to study the dynamics of a mass-spring network in which access is limited.

\begin{example}\label{ex:spring0}
Consider the mass-spring network illustrated in figure~\ref{fig:spring}, with nodes at locations $x_i$, $i=1,2,3,4$ lying on a line and edges representing springs between nodes. For simplicity we assume that all the springs have the same spring constant ($k=1$) and that all the nodes have unit mass. (The precise position of the nodes on the line does not matter for this discussion.)

Suppose each node $x_i$ is subject to a time harmonic displacement $u_i(\omega) e^{j\omega t}$ with frequency $\omega$ in the direction of the line and $j = \sqrt{-1}$. Then the resulting force at node $x_i$ is also time harmonic in the direction of the line and is of the form $f_i(\omega) e^{j \omega t}$. Writing the balance of forces acting on each node with the laws of motion, one can show that the vector of forces $\bbf(\omega) = [ f_1(\omega), \ldots, f_4(\omega)]^T$ is linearly related to the vector of displacements $\bu(\omega) = [ u_1(\omega), \ldots, u_4(\omega)]^T$ by the equation
\begin{equation}
  \bbf(\omega) =  (K - \omega^2 I)\bu(\omega).
\end{equation}
Here the matrix $K$ is the {\em stiffness} matrix
\[
 K =
 \begin{bmatrix}
 1 & -1\\
 -1 & 2 & -1 \\
 & -1 & 2 & -1\\
 && -1 & 1
 \end{bmatrix}.
\]
If we let $\lambda \equiv \omega^2$, we see that the eigenmodes of the stiffness matrix $K$ correspond to non-zero displacements that do not generate forces. For instance, the eigenmode corresponding to the zero frequency is $\bu = [1,1,1,1]^T$, i.e. by displacing all nodes by the same amount, there are no net forces at the nodes.

Suppose we only have access to certain terminal (or boundary) nodes of this network, say $\cB = \{1,4\}$. Then we can write the equilibrium of forces at the interior nodes $\cI = \{2,3\}$ and conclude that the net forces $\bbf_\cB$ at the terminal nodes depend linearly on the displacements $\bu_\cB$ at the terminal nodes according to the equation
\begin{equation}
 \bbf_\cB(\omega) = ( R_{\omega^2} ( K; \cB) - \omega^2 I) \bu_\cB(\omega).
\end{equation}

The spectrum and inverse spectrum of the response are
$$\sigma(R_{\omega^2} ( K; \cB))=\{2\pm\sqrt{2},2,0\} \text{and} \sigma^{-1}(R_{\omega^2} ( K; \cB))=\{3,1\}.$$ The eigenvalues of $R_{\omega^2} ( K; \cB)$ correspond to frequencies for which there is a displacement of the boundary nodes $\cB$ that generate no forces at these nodes. Conversely, the resonances (or inverse eigenvalues) of $R_{\omega^2} ( K; \cB)$ correspond to frequencies at which there is a displacement of the boundary nodes for which the resulting forces are infinitely large.
\end{example}

\begin{figure}
 \begin{center}
 \includegraphics[width=8cm]{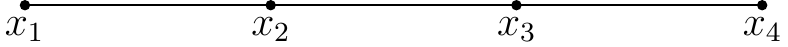}
 \end{center}
 \caption{The mass spring network of example~\ref{ex:spring0} with boundary nodes $\mathcal{B}=\{1,4\}$ and interior nodes $\mathcal{I}=\{2,3\}$.}
 \label{fig:spring}
\end{figure}

\subsection{Sequential Reductions}
In the previous section we observed that the isospectral reduction $R(M;\mathcal{B})$ of $M\in\BW^{n\times n}_{\pi}$ is again a matrix in $\BW^{m\times m}_{\pi}$. It is therefore possible to reduce the matrix $R(M;\mathcal{B})$ again over some subset of $\mathcal{B}$. That is, we may sequentially reduce the matrix $M$. However, a natural question is to what extent does a sequentially reduced matrix depends on the particular sequence of index sets over which it has been reduced.

As it turns out, if a matrix has been reduced over the index set $\mathcal{B}_1$ then $\mathcal{B}_2$ up to the index set $\mathcal{B}_m$ then the resulting matrix depends only on the index set $\mathcal{B}_m$. To formalize this, let $M\in\mathbb{W}^{n\times n}_{\pi}$ and suppose there are non-empty sets $\mathcal{B}_1,\dots,\mathcal{B}_m$ such that $N\supset\mathcal{B}_1\supset,\dots,\supset\mathcal{B}_m$. Then $M$ can be \emph{sequentially reduced} over the sets $\mathcal{B}_1,\dots,\mathcal{B}_m$ where we write
$$R_{\lambda}(M;\mathcal{B}_1,\dots,\mathcal{B}_m)=
R_{\lambda}\big(\dots R_{\lambda}(R_{\lambda}(M;\mathcal{B}_1);\mathcal{B}_2)\dots;\mathcal{B}_m\big).$$
If $M$ is sequentially reduced over the index sets $\mathcal{B}_1,\dots,\mathcal{B}_m$ we call $\mathcal{B}_m$ the \emph{final index set} of this sequence of reductions.

\begin{theorem}\label{theorem2}\textbf{(Uniqueness of Sequential Reductions)}
For $M(\lambda)\in\mathbb{W}^{n\times n}_{\pi}$ suppose $N\supset\mathcal{B}_1\supset,\dots,\supset\mathcal{B}_m$ where $\mathcal{B}_m$ is non-empty. Then
$$R_{\lambda}(M;\mathcal{B}_1,\dots,\mathcal{B}_m)=R_{\lambda}(M;\mathcal{B}_m).$$
\end{theorem}

That is, in a sequence of reductions the resulting matrix is completely specified by the final index set. To prove theorem \ref{theorem2} we first require the following lemma.

\begin{lemma}\label{lemma2}
Let the non-empty sets $\mathcal{B}$, $\mathcal{I}$, and $\mathcal{J}$ partition $N$. If $M(\lambda) \in\BW^{n\times n}_{\pi}$ then $R_{\lambda}(M;\mathcal{B}\cup\mathcal{I},\mathcal{B})=R_{\lambda}(M;\mathcal{B})$.
\end{lemma}

\begin{proof}
Assume without loss of generality that $M \in \BW^{n\times n}_{\pi}$ can be written as
\[
 M(\lambda)
 =
 \begin{bmatrix}
  M_{\mathcal{BB}} & M_{\mathcal{BI}} & M_{\mathcal{BJ}}\\
  M_{\mathcal{IB}} & M_{\mathcal{II}} & M_{\mathcal{IJ}}\\
  M_{\mathcal{JB}} & M_{\mathcal{JI}} & M_{\mathcal{JJ}}
 \end{bmatrix}.
\]
Using the definition of isospectral reduction we have
\begin{equation}\label{eq:smb}
 R_{\lambda}(M;\mathcal{B})
 = M_{\mathcal{BB}}
 -
 \begin{bmatrix} M_{\mathcal{BI}} & M_{\mathcal{BJ}} \end{bmatrix}
 \begin{bmatrix}
  M_{\mathcal{II}} - \lambda I &  M_{\mathcal{IJ}}\\
  M_{\mathcal{JI}} & M_{\mathcal{JJ}} - \lambda I
 \end{bmatrix}^{-1}
 \begin{bmatrix} M_{\mathcal{IB}} \\ M_{\mathcal{JB}} \end{bmatrix} \ \ \text{and}
\end{equation}

\begin{equation}\label{eq:smbi}
 R_{\lambda}(M;\mathcal{B}\cup\mathcal{I})
 =
 \begin{bmatrix}
  M_{\mathcal{BB}} & M_{\mathcal{BI}} \\
  M_{\mathcal{IB}} & M_{\mathcal{II}}
 \end{bmatrix}
 -
 \begin{bmatrix} M_{\mathcal{BJ}} \\ M_{\mathcal{IJ}} \end{bmatrix}
 (M_{\mathcal{JJ}} -  \lambda I)^{-1}
 \begin{bmatrix} M_{\mathcal{JB}} & M_{\mathcal{JI}} \end{bmatrix}.
\end{equation}
Taking the isospectral reduction of $R_{\lambda}(M;\mathcal{B}\cup\mathcal{I})$ over $\mathcal{B}$ in \eqref{eq:smbi} we have
\begin{multline}\label{eq:ssmbib}
 R_{\lambda}(M;\mathcal{B}\cup\mathcal{I},\mathcal{B}) =
 M_{\mathcal{BB}} - M_{\mathcal{BJ}} K(\lambda)^{-1} M_{\mathcal{JB}}\\
 -\left[ (M_{\mathcal{BI}} - M_{\mathcal{BJ}} K(\lambda)^{-1} M_{\mathcal{JI}})
        T(\lambda)^{-1}
        (M_{\mathcal{IB}} - M_{\mathcal{IJ}} K(\lambda)^{-1} M_{\mathcal{JB}}) \right],
\end{multline}
where $K(\lambda) \equiv M_{\mathcal{JJ}} - \lambda I$ and $T(\lambda) \equiv M_{\mathcal{II}} - \lambda I - M_{\mathcal{IJ}} K(\lambda)^{-1} M_{\mathcal{JI}}$. Note that both $K(\lambda)^{-1}$ and $T(\lambda)^{-1}$ exist following the proof of lemma \ref{lemma1}. To show the desired result we need to verify that expressions \eqref{eq:smb} and \eqref{eq:ssmbib} are equal.

Recall the following identity for the inverse of a square matrix $M$ with $2\times 2$ blocks:
\begin{equation}\label{eq:2x2inv}
 M^{-1} =
 \begin{bmatrix}
 A & B\\
 C & D
 \end{bmatrix}^{-1}
 =
 \begin{bmatrix}
  E^{-1} & -E^{-1} B D^{-1}\\
  -D^{-1} C E^{-1} & D^{-1} + D^{-1} C E^{-1} B D^{-1}
 \end{bmatrix},
\end{equation}
where $E = A - B D^{-1} C$ is the Schur complement of $D$ in $M$. The determinantal identity \eqref{eq:detschur} implies  that $M$ is invertible if and only if $D$ and $E$ are invertible. Using \eqref{eq:2x2inv} to find the inverse of the $2 \times 2$ block matrix appearing in \eqref{eq:smb} we get
\begin{equation} \label{eq:IJblock}
 \begin{bmatrix}
  M_{\mathcal{II}} - \lambda I &  M_{\mathcal{IJ}}\\
  M_{\mathcal{JI}} & M_{\mathcal{JJ}} - \lambda I
 \end{bmatrix}^{-1}=
\end{equation}
\begin{equation*}
 \begin{bmatrix}
  T(\lambda)^{-1} & - T(\lambda)^{-1} M_{\mathcal{IJ}} K(\lambda)^{-1}\\
  - K(\lambda)^{-1} M_{\mathcal{JI}} T(\lambda)^{-1} &
 K(\lambda)^{-1} + K(\lambda)^{-1} M_{\mathcal{JI}} T(\lambda)^{-1} M_{\mathcal{IJ}} K(\lambda)^{-1}
 \end{bmatrix}.
\end{equation*}
Using \eqref{eq:IJblock} in \eqref{eq:smb} we get \eqref{eq:ssmbib} completing the proof.
\end{proof}

We now give a proof of theorem \ref{theorem2}.

\begin{proof}
For $M\in\mathbb{W}^{n\times n}_{\pi}$ suppose $N\subset\mathcal{B}_1\supset\dots\supset\mathcal{B}_m$ where $\mathcal{B}_m\neq\emptyset.$ If $m=2$ then lemma \ref{lemma2} directly implies that $R_{\lambda}(M;\mathcal{B}_1,\mathcal{B}_2)=R_{\lambda}(M;\mathcal{B}_2)$. For $2\leq k<m$ suppose $R_{\lambda}(M;\mathcal{B}_1,\dots,\mathcal{B}_k)=R_{\lambda}(M;\mathcal{B}_k)$. Then
$$R_{\lambda}(M;\mathcal{B}_1,\dots,\mathcal{B}_k,\mathcal{B}_{k+1})=R_{\lambda}(M;\mathcal{B}_k,\mathcal{B}_{k+1}) =R_{\lambda}(M;\mathcal{B}_{k+1})$$
where the second equality follows from lemma \ref{lemma2}. By induction it then follows that $R_{\lambda}(M;\mathcal{B}_1,\dots,\mathcal{B}_m)=R_{\lambda}(M;\mathcal{B}_m)$.
\end{proof}

\begin{example}
Let $M\in\complex^{4\times 4}$ be the matrix given by
$$M=\left[\begin{array}{cccc}
1&0&1&0\\
0&1&0&1\\
0&1&1&1\\
1&0&1&1
\end{array}\right]$$
and let $\mathcal{B}=\{1,2\}$. Our goal in this example is to illustrate that $$R_{\lambda}(M;\mathcal{B}) = R_{\lambda}(M;\mathcal{B}\cup \{3\},\mathcal{B})= R_{\lambda}(M;\mathcal{B}\cup \{4\},\mathcal{B}).$$ As one can compute
$$R_{\lambda}(M;\mathcal{B}\cup \{3\})=
\begin{bmatrix}
1&0&1\\
\frac{1}{\lambda-1}&1&\frac{1}{\lambda-1}\\
\frac{1}{\lambda-1}&1&\frac{\lambda}{\lambda-1}
\end{bmatrix} \ \ \text{and} \ \
R_{\lambda}(M;\mathcal{B}\cup \{4\})=\left[\begin{array}{ccc}
1&\frac{1}{\lambda-1}&\frac{1}{\lambda-1}\\
0&1&1\\
1&\frac{1}{\lambda-1}&\frac{\lambda}{\lambda-1}
\end{array}\right].$$
Although $R_{\lambda}(M;\mathcal{B}\cup \{3\})\neq R_{\lambda}(M;\mathcal{B}\cup \{4\})$, note that by reducing both of these matrices over $\mathcal{B}=\{1,2\}$ one has $$R_{\lambda}(M;B) = R_{\lambda}(M;\mathcal{B}\cup \{3\},\mathcal{B})= R_{\lambda}(M;\mathcal{B}\cup \{4\},\mathcal{B})=
\begin{bmatrix}
\frac{\lambda^2-2\lambda+1}{\lambda^2-2\lambda}&\frac{\lambda-1}{\lambda^2-2\lambda}\\
\frac{\lambda-1}{\lambda^2-2\lambda}&\frac{\lambda^2-2\lambda+1}{\lambda^2-2\lambda}
\end{bmatrix}.$$
As a final observation, we note that $\sigma(M)=\{\frac{1}{2}(3\pm\sqrt{5}),\frac{1}{2}(1\pm\sqrt{-3})\}$ and $\sigma(M_{II})=\{0,2\}$ for $I = \{3,4\}$. Hence the matrix $M$ and the reduced matrix $R(M,\mathcal{B})$ have the same eigenvalues by corollary~\ref{cor1}. That is, an isospectral reduction need not have any effect on the spectrum of a matrix. (In this example the inverse spectrum does change with the reduction).
\end{example}

\subsection{Spectral Inverse} \label{sec:spinv}
Although a matrix $M\in\mathbb{W}(\lambda)^{n\times n}$ has both a spectrum and an inverse spectrum, the techniques that have been developed to analyze its spectral properties have been restricted to its spectrum \cite{Bunimovich:2011:IGR,Bunimovich:2012:IGT,Bunimovich:2012:IC}. The goal in this section is to introduce a new matrix transformation that exchanges a matrix' spectrum and inverse spectrum. This transformation allows us to investigate the inverse spectrum of these matrices with tools meant to study its spectrum. Additionally, we use this transformation to define the inverse pseudospectrum (or pseudoresonances) of a matrix from the pseudospectrum of a matrix (Section \ref{sec:psi}).

\begin{definition}
For $M(\lambda)\in\mathbb{W}^{n\times n}$ let $S_{\lambda}(M)\in\mathbb{W}^{n\times n}$ be the matrix $$S_{\lambda}(M)=(M(\lambda)-\lambda I)^{-1}+\lambda I\in\mathbb{W}^{n\times n},$$ if the inverse $(M(\lambda)-\lambda I)^{-1}$ exists. The matrix $S_{\lambda}(M)$ is called the \emph{spectral inverse} of the matrix $M(\lambda)$.
\end{definition}

We typically write the spectral inverse of $M\in\mathbb{W}^{n\times n}$ as $S(M)$ unless otherwise needed. We also observe that not every matrix $M\in\BW^{n\times n}$ has a spectral inverse. For instance, the matrix
$$M=\left[
\begin{array}{cc}
\lambda&0\\
0&\lambda
\end{array}
\right]$$
cannot be spectrally inverted. However, if $M$ has a spectral inverse then the following holds.

\begin{theorem}\label{theorem3}
Suppose $M(\lambda)\in\BW^{n\times n}$ has a spectral inverse $S(M)$. Then
$$\sigma\big(S(M)\big)=\sigma^{-1}(M) \ \ \text{and} \ \ \sigma^{-1}\big(S(M)\big)=\sigma(M).$$
\end{theorem}

\begin{proof}
Let $M(\lambda)\in\BW^{n\times n}$ with spectral inverse $S(M)$. Note that
$$\det\big((S(M)-\lambda I)(M-\lambda I)\big)=\det\big((M-\lambda I)^{-1}(M-\lambda I)\big)=\det(I)=1.$$
As the determinant is multiplicative then
$$\det(S(M)-\lambda I)=\det(M-\lambda I)^{-1},$$
and the result follows.
\end{proof}

A matrix $M\in\mathbb{W}^{n\times n}$ may or may not have a spectral inverse. However, if $M\in\mathbb{W}^{n\times n}_{\pi}$ then the proof of lemma \ref{lemma1} implies that $M-\lambda I$ is invertible. Therefore, $S(M)$ exists. This result is stated in the following lemma.

\begin{lemma}\label{lemma3}
If $M(\lambda)\in\mathbb{W}^{n\times n}_{\pi}$, then $M(\lambda)$ has a spectral inverse.
\end{lemma}

\begin{example}\label{ex:5}
Let $M\in\BW^{4\times 4}_{\pi}$ be the matrix given by
$$M=\left[
\begin{array}{cccc}
\frac{1}{\lambda}&\frac{1}{\lambda}&0&0\\
0&\frac{1}{\lambda}&1&0\\
0&0&\frac{1}{\lambda}&0\\
0&0&0&\frac{1}{\lambda}
\end{array}
\right]$$
for which we have
$$\det\big(M(\lambda)-\lambda I\big)=\frac{\lambda^8-4\lambda^6+6\lambda^4-4\lambda^2+1}{\lambda^4}.$$
As one can calculate, the spectral inverse $S(M)$ is the matrix
$$S(M)=\left[
\begin{array}{cccc}
\frac{-\lambda}{\lambda^2-1}&\frac{-\lambda}{(\lambda^2-1)^2}&\frac{-\lambda^2}{(\lambda^2-1)^3}&\frac{-\lambda^3}{(\lambda^2-1)^4}\\
0&\frac{-\lambda}{\lambda^2-1}&\frac{-\lambda^2}{(\lambda^2-1)^2}&\frac{-\lambda^3}{(\lambda^2-1)^3}\\
0&0&\frac{-\lambda}{\lambda^2-1}&\frac{-\lambda^2}{(\lambda^2-1)^2}\\
0&0&0&\frac{-\lambda}{\lambda^2-1}
\end{array}
\right]+\lambda I.$$
Taking the determinant of $S(M)-\lambda I$ one has
$$\det\big(S(M)-\lambda I\big)=\frac{\lambda^4}{\lambda^8-4\lambda^6+6\lambda^4-4\lambda^2+1}.$$
That is, $\det\big(S(M)-\lambda I\big)=\det(M(\lambda)-\lambda I)^{-1}$.
\end{example}

Observe, that for any $M\in\mathbb{W}_{\pi}^{n\times n}$ the spectral inverse $S(M)\notin\mathbb{W}_{\pi}^{n\times n}$. Therefore, we have no guarantee that $S(M)$ can be isospectrally reduced. However, the following holds.

\begin{theorem}\label{theorem4}\textbf{(Reductions of the Spectral Inverse)}
For $M(\lambda)\in\mathbb{W}^{n\times n}_{\pi}$ suppose $N\supset\mathcal{B}_1\supset,\dots,\supset\mathcal{B}_m$ where $\mathcal{B}_m$ is non-empty. Then
\begin{enumerate}[(i)]
\item $R_{\lambda}\big(S(M);\mathcal{B}_m\big)$ exists;
\item $R_{\lambda}(S(M);\mathcal{B}_1,\dots,\mathcal{B}_m)=R_{\lambda}(S(M);\mathcal{B}_m)$; and
\item $R_{\lambda}(S(M);\mathcal{B})=(M-\lambda I)^{-1}/\left[(M-\lambda I)^{-1}\right]_{\mathcal{II}}+\lambda I$ where $\mathcal{I}=N-\mathcal{B}_m$.
\end{enumerate}
\end{theorem}

\begin{proof}
For $M \in\BW^{n\times n}_{\pi}$ suppose $\mathcal{B}$ and $\mathcal{I}$ form a non-empty partition of $N$. By lemmas~\ref{lemma1} and \ref{lemma3}, the matrix $S(M)$ exists and
\[
 S(M) - \lambda I = (M-\lambda I)^{-1} \in \BW_\pi^{n \times n}.
\]
Equating blocks in the previous equation gives that the matrices $[S(M)]_{\cB\cB} - \lambda I$, $[S(M)]_{\cB\cI}$, $[S(M)]_{\cI\cB}$ and $[S(M)]_{\cI\cI} - \lambda I$ all have entries in $\BW_\pi$. Moreover $[S(M)]_{\cI\cI} - \lambda I$ is not identically zero so its inverse exists. We deduce that the reduction of $S(M)$ exists and is
\[
 \begin{aligned}
 R_\lambda(S(M); \cB) -\lambda I & = ([S(M)]_{\cB\cB} - \lambda I) - [S(M)]_{\cB\cI} \left([S(M)]_{\cI\cI} - \lambda I \right)^{-1} [S(M)]_{\cI\cB} \\ &\in \BW_\pi^{|\cB| \times |\cB|}.
 \end{aligned}
\]
To prove (iii), simply notice that $[S(M)]_{\cB\cB} - \lambda I = \left[(M-\lambda I)^{-1} \right]_{\cB\cB}$, $[S(M)]_{\cI\cB} = [S(M) - \lambda I]_{\cI \cB} = \left[(M-\lambda I)^{-1} \right]_{\cI\cB}$, $[S(M)]_{\cB\cI} =\left[(M-\lambda I)^{-1} \right]_{\cB\cI}$ and $[S(M)]_{\cI\cI} - \lambda I = \left[(M-\lambda I)^{-1} \right]_{\cI\cI}$. These relations imply (iii).

Substituting each submatrix $M_{\mathcal{RC}}$ in the proof of lemma 2 by the matrix
$$S(M)_{\mathcal{RC}}=
\begin{cases}
(M-\lambda I)^{-1}_{\mathcal{RC}}+\lambda I \ &\text{if} \ \mathcal{R}=\mathcal{C},\\
(M-\lambda I)^{-1}_{\mathcal{RC}} \ &\text{otherwise}
\end{cases}$$
and then following the proof of theorem \ref{theorem2} using $S(M)$ instead of $M$ yields a proof of part (ii).
\end{proof}

Theorem \ref{theorem4} states that any matrix $M\in\mathbb{W}^{n\times n}_{\pi}$ has a spectral inverse and that this inverse can be reduced over any index set. Observe the similarity between equations (\ref{eq:sch}) and part (iii) of theorem \ref{theorem4}.

\subsection{Gershgorin-Type Estimates}

If $M\in\mathbb{W}^{n\times n}$ then its inverse spectrum $\sigma^{-1}(M)$ are the complex numbers at which the determinant $\det(M-\lambda I)$ is undefined. Since the determinant of a matrix is composed of various products and sums of its entries then equations (\ref{eq:add}) and (\ref{eq:mult}) imply the following proposition. Hereinafter for $A \subset \complex$, the set $\overline{A}$ is the complement of $A$ in $\complex$.

\begin{proposition}\label{prop0}
If $M(\lambda)\in\mathbb{W}^{n\times n}$ then $\sigma^{-1}(M)\subseteq \overline{\dom(M)}$.
\end{proposition}

Phrased another way, the inverse eigenvalues of a matrix $M\in\mathbb{W}^{n\times n}$ are complex numbers at which the matrix $M$ is undefined, i.e. in the complement of $\dom(M)$. However, it is not always the case that the converse holds as the following example demonstrates.

\begin{example}
Consider the reduced matrix
$$R(M;\mathcal{B})=\left[
\begin{array}{cc}
\frac{1}{\lambda-1}&\frac{1}{\lambda-1}\\
\frac{1}{\lambda}&\frac{\lambda+1}{\lambda}
\end{array}
\right]\in\BW^{2\times 2}
$$
found in example \ref{ex:2}. As computed in example \ref{ex:3}, we have $\sigma^{-1}(R(M;\mathcal{B}))=\emptyset$ and yet $\overline{\dom(R(M;\mathcal{B}))} = \{ 0,1 \}$.
\end{example}

To improve upon proposition \ref{prop0} we look for methods of estimating the inverse spectrum of a matrix. The following well-known theorem due to Gershgorin gives a simple method for approximating the eigenvalues of a square matrix with complex valued entries.

\begin{theorem}{\textbf{(Gershgorin \cite{Gershgorin:1931:UDA})}}\label{theorem:gersh}
Let $M\in\mathbb{C}^{n\times n}$. Then all eigenvalues of $M$ are contained in the set
$$\Gamma(M)=\bigcup^n_{i=1}\big\{\lambda\in \mathbb{C}:|\lambda-M_{ii}|\leq \sum_{j=1,j\neq i}^n |M_{ij}|\big\}.$$
\end{theorem}

In \cite{Bunimovich:2011:IGR} it was shown that Gershgorin's theorem can be extended to matrices $M\in\mathbb{W}^{n\times n}$. Our goal in this Section is to further extend this result by using the spectral inverse introduced in Section~\ref{sec:spinv} to estimate the inverse spectrum (or resonances) of matrix $M\in\mathbb{W}^{n\times n}$. To do so we first define the notion of a \textit{polynomial extension} of the matrix $M$.

\begin{definition}
For $M(\lambda)\in\mathbb{W}^{n\times n}$ with entries $M_{ij}=p_{ij}/q_{ij}$ let $L_i(M)$=$\prod_{j=1}^{n}q_{ij}$ for $1\leq i\leq n$. We call the matrix $\overline{M}$ given by
$$\overline{M}_{ij}=\begin{cases}
L_iM_{ij} \hspace{.85in} i\neq j\\
L_i\big(M_{ij}-\lambda\big)+\lambda \hspace{0.2in} i=j
                         \end{cases}, \ \ 1\leq i,j\leq n
$$ the \textit{polynomial extension} of $M$.
\end{definition}

Note that for any $M\in\mathbb{W}^{n\times n}$ the matrix $\overline{M}\in\mathbb{C}[\lambda]^{n\times n}$. The following theorem extends Gershgorin's original theorem to matrices in $\mathbb{W}^{n\times n}$ (see theorem 3.4 in \cite{Bunimovich:2011:IGR}).

\begin{theorem}\label{theorem5}
Let $M(\lambda)\in\mathbb{W}^{n\times n}$. Then $\sigma(M)$ is contained in the set
$$\Gamma(M)=\bigcup_{i=1}^n\big\{\lambda\in\mathbb{C}:|\lambda-\overline{M}_{ii}|\leq\sum_{j=1,j\neq i}^n |\overline{M}_{ij}|\big\}.$$
\end{theorem}

We call the set $\Gamma(M)$ the \emph{Gershgorin-type region} of the matrix $M$ or simply its Gershgorin region. (The notation in \cite{Bunimovich:2011:IGR} is $\mathcal{BW}_{\Gamma}(M)$).

An immediate corollary to theorem \ref{theorem3} and theorem \ref{theorem5} is the following.

\begin{corollary}\label{cor2}
Let $M(\lambda)\in\mathbb{W}^{n\times n}$. Then $\sigma^{-1}(M)$ is contained in the set
$$\Gamma\big(S(M)\big)=\bigcup_{i=1}^n\big\{\lambda\in\mathbb{C}: |\lambda-\overline{S(M)}_{ii}|\leq\sum_{j=1,j\neq i}^n |\overline{S(M)}_{ij}|\big\}.$$
\end{corollary}

\begin{example}\label{ex:6}
Let $M\in\mathbb{W}^{4\times 4}$ be the matrix considered in example \ref{ex:5}. Then
$$\overline{S(M)}=\left[
\begin{array}{cccc}
-\lambda(\lambda^2-1)^9&-\lambda(\lambda^2-1)^8&-\lambda^2(\lambda^2-1)^7&-\lambda^3(\lambda^2-1)^6\\
0&-\lambda(\lambda^2-1)^5&-\lambda^2(\lambda^2-1)^4&-\lambda^3(\lambda^2-1)^3\\
0&0&-\lambda(\lambda^2-1)^2&-\lambda^2(\lambda^2-1)^1\\
0&0&0&-\lambda
\end{array}
\right]+\lambda I.$$
The region $\Gamma\big(S(M)\big)$ is shown in figure \ref{fig1} (left).

\begin{figure}
\begin{center}
\begin{tabular}{cc}
 \includegraphics[width=0.49\textwidth]{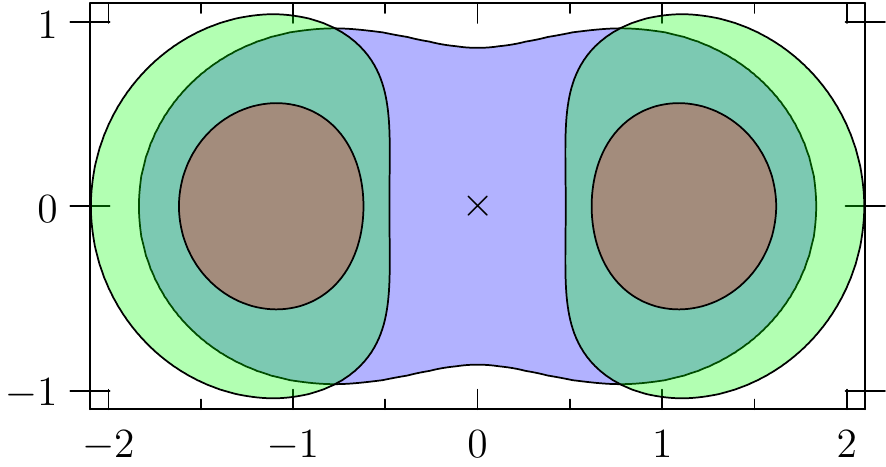} &
 \includegraphics[width=0.49\textwidth]{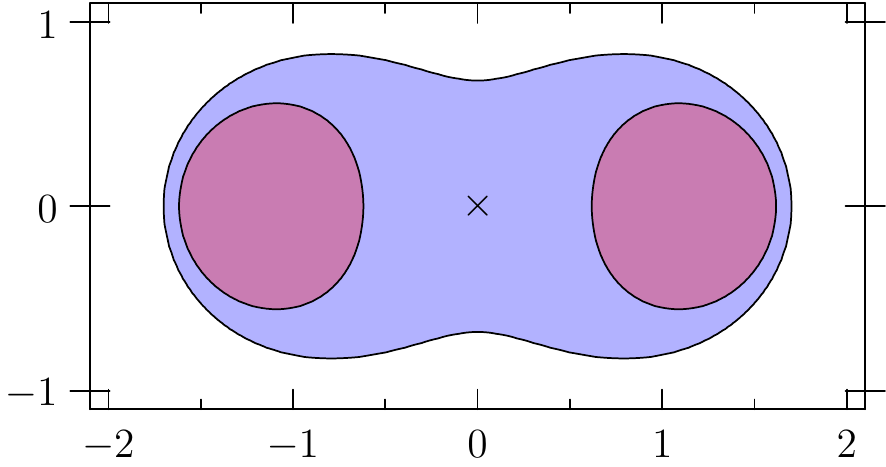}
\end{tabular}
\end{center}
\caption{Left: $\Gamma(S(M))$. Right: $\Gamma(R(S(M);\mathcal{B}))$ where the inverse spectrum $\sigma^{-1}(M)=\{0,0,0,0\}$ is indicated by a ``$\times$''.}\label{fig1}
\end{figure}

We note that the Gershgorin-type region $\Gamma(S(M))$ is the union of the sets
$$\Gamma(S(M))_i=\big\{\lambda\in\mathbb{C}: |\lambda-\overline{S(M)}_{ii}|\leq\sum_{j=1,j\neq i}^n |\overline{S(M)}_{ij}|\big\},$$ for $i=1,2,3,4$. The regions
$\Gamma(S(M))_1$, $\Gamma(S(M))_2$, and $\Gamma(S(M))_3$ in figure \ref{fig1}  are shown in blue, green, and red respectively. Transparency is used to highlight the intersections.  The same strategy is used in Section \ref{sec:psi} to display pseudospectra (or inverse pseudospectra) of a matrix.The set $\Gamma(S(M))_4=\{0\}$ is contained in the inverse spectrum $\sigma^{-1}(M)=\{0,0,0,0\}$, which is indicated in the figure.
\end{example}

One of the main results of \cite{Bunimovich:2011:IGR} is that the Gershgorin region of a reduced matrix $R(M;\mathcal{B})$ is a subset of the Gershgorin region of the unreduced matrix $M\in\mathbb{W}^{n\times n}$ (see theorem 5.1 \cite{Bunimovich:2011:IGR}). In the same way the inverse eigenvalue estimates given in corollary \ref{cor2} can be improved via the process of isospectral matrix reduction.

\begin{theorem}{\textbf{(Improved Inverse Eigenvalue Estimates)}}\label{impgersh}
Let $M(\lambda)\in\mathbb{W}^{n\times n}_{\pi}$ where $\mathcal{B}$ is any nonempty subset of $N$. Then $$\Gamma\big(R(S(M);\mathcal{B})\big)\subseteq\Gamma\big(S(M)\big).$$
\end{theorem}

A proof of theorem \ref{impgersh} can be obtained by following the proof of theorem 5.1 in \cite{Bunimovich:2011:IGR} and by using theorem \ref{theorem4}(ii).

\begin{example}\label{ex:7}
Let $M\in\mathbb{W}^{4\times 4}$ be the matrix given in example \ref{ex:5}. For the index set $\mathcal{B}=\{1,2,3\}$, the reduction of the spectral inverse of $M$ is 
$$R(S(M);\mathcal{B})=\left[
\begin{array}{ccc}
\frac{-\lambda}{\lambda^2-1}&\frac{-\lambda}{(\lambda^2-1)^2}&\frac{-\lambda^2}{(\lambda^2-1)^3}\\
0&\frac{-\lambda}{\lambda^2-1}&\frac{-\lambda^2}{(\lambda^2-1)^2}\\
0&0&\frac{-\lambda}{\lambda^2-1}.
\end{array}
\right]+\lambda I$$
Its polynomial extension follows:
$$\overline{R(S(M);\mathcal{B})}=\left[
\begin{array}{ccc}
-\lambda(\lambda^2-1)^9&-\lambda(\lambda^2-1)^8&-\lambda^2(\lambda^2-1)^7\\
0&-\lambda(\lambda^2-1)^5&-\lambda^2(\lambda^2-1)^4\\
0&0&-\lambda(\lambda^2-1)^2
\end{array}
\right]+\lambda I.$$
The Gershgorin-type region of the reduced matrix $R(S(M);\mathcal{B})$ is shown in figure \ref{fig1} (right) where one can see that $\sigma^{-1}(M)\subset\Gamma\big(R(S(M);\mathcal{B})\big)\subset\Gamma\big(S(M)\big).$ The regions $\Gamma(R(S(M);\cB))_1$ and $\Gamma(R(S(M);\cB))_2$ are in blue and red, respectively.

\end{example}

\begin{remark}
In this section we have considered how Gershgorin-type estimates can be used to estimate the inverse spectrum of a matrix $M\in\mathbb{W}^{n\times n}$. We note that the same is true of the eigenvalue estimates associated with Brauer, Brualdi, and Varga (see \cite{Bunimovich:2011:IGR} for details).
\end{remark}

\section{Pseudospectra and pseudoresonances}\label{sec:psi}
A pseudospectrum of a matrix $M\in\complex^{n \times n}$ is essentially the collection of scalars that behave, to within a given tolerance, as an eigenvalue of $M$. These values indicate to what extent the eigenvalues of the matrix $M$ are stable under perturbation of the matrix entries. See e.g. \cite{Trefethen:2005:SP} for a review of pseudospectra including their history and applications.

We first extend the notion of pseudospectra to matrices in $\BW^{n\times n}_\pi$. Then we show that the spectral inverse of a matrix can be used to define {\em inverse pseudospectra} for matrices in $\BW^{n\times n}_\pi$. The inverse pseudospectra or \emph{pseudoresonances} of $M$ are the scalars that behave, to within a certain tolerance, as inverse eigenvalues or resonances of $M$.

We study pseudoresonances and their relation to pseudospectra in Section~\ref{sec:psire}. In Section~\ref{sec:redpsi} we show that an isospectral reduction shrinks the pseudospectrum of matrix for a given tolerance. Throughout this discussion we consider the simple mass-spring network introduced in Section \ref{sec:imr} to give a physical interpretation to these concepts.

Before formally extending the notion of pseudospectra to matrices in $\mathbb{W}^{n\times n}$ we note that pseudospectra has been previously generalized to matrix polynomials in \cite{Lancaster:2005:PMP,Boulton:2008:PMP}.

\subsection{Pseudospectra} \label{sec:psisp}
For a matrix $A\in\mathbb{C}^{n\times n}$, if $\lambda\in\sigma(A)$ then there is always at least one eigenvector $\textbf{v}\in\mathbb{C}^n$ of $A$ associated with $\lambda$. However, recall from Section \ref{sec:mat} that a matrix $M(\lambda)\in\mathbb{W}_{\pi}^{n\times n}$ may have an eigenvalue $\lambda_0$ for which $M(\lambda_0)$ is undefined. This may seem problematic especially if we would like to find an eigenvector associated with $\lambda_0$. In fact, it is still possible to do so.

Assuming $\lambda_0$ is a solution to the equation $\det(M(\lambda)-\lambda I)=0$, the standard theory of linear algebra implies that there is a vector $\textbf{v}$ such that when the product $(M(\lambda)-\lambda I)\textbf{v}$ is evaluated at $\lambda=\lambda_0$, the result is the zero vector. Keeping this sequence in mind, we define the product of a matrix and vector as follows. For any $M(\lambda)\in\mathbb{W}^{n\times n}_{\pi}$ and $\textbf{v}\in\mathbb{C}^n$ we let the product
$$(M(\lambda)-\lambda I)\textbf{v}\equiv (M(s) - s I) \bv|_{s=\lambda}.$$
This definition allows us to associate an eigenvectors to each eigenvalue of a matrix $M(\lambda)\in\mathbb{W}_{\pi}^{n\times n}$. To demonstrate this idea we give the following example.

\begin{example}
Consider the matrix $M(\lambda)\in\mathbb{W}^{2\times 2}$ given by
$$M(\lambda)=
\left[
\begin{array}{cc}
1&\frac{1}{\lambda-1}\\
0&1
\end{array}
\right].
$$
Here, one can readily see that $\sigma(M)=\{1,1\}$. Although $M(1)$ is undefined, the vector $\mathbf{v}=[1 \ 0]^T$ has the property
$$(M(1)- 1I)\mathbf{v}=
\left[
\begin{array}{cc}
1-s&\frac{1}{s-1}\\
0&1-s
\end{array}
\right]
\left[
\begin{array}{c}
1\\
0
\end{array}
\right]\Big|_{s=1}=
\left[
\begin{array}{c}
0\\
0
\end{array}
\right].
$$
By definition the vector $\mathbf{v}$ is an eigenvector associated with the eigenvalue $1$ despite the fact that $M(\lambda)$ is not defined for $\lambda=1$.

Importantly, for the vector norm $||\cdot||$ we have
$$||(M(\lambda)- \lambda I)\mathbf{v}||=
\left\|\left[
\begin{array}{c}
1-\lambda\\
0
\end{array}
\right]\right\|.
$$
Hence, the size of $(M(\lambda)-\lambda I)\mathbf{v}$ varies continuously with respect to $\lambda$ even where $M(\lambda)$ is undefined. This is useful since we study values of $\lambda$ that act almost like eigenvalues of $M(\lambda)$.
\end{example}

Suppose that for a given tolerance $\epsilon>0$, there is a scalar $\lambda \in \complex$ and a unit vector $\bv\in\complex^n$ for which  $\| (M(\lambda) - \lambda I) \bv\| < \eps$. If this is the case then the vector $\bv$ is said to be an {\em $\eps$-pseudoeigenvector} of the matrix $M(\lambda)$ corresponding to the {\em $\eps$-pseudoeigenvalue} $\lambda$. The $\epsilon-$pseudospectrum of $M(\lambda)$ is defined as the set of all such $\lambda$. We state this and two other equivalent definitions of the $\epsilon-$pseudospectrum below. For $\Omega\subset\mathbb{C}$, let $cl(\Omega)$ be the closure of $\Omega$ in $\mathbb{C}$.

\begin{definition}\label{def:psisp}
Let $\epsilon>0$. The {\em $\epsilon$-pseudospectrum} of $M(\lambda) \in\mathbb{W}^{n\times n}_\pi$ is defined equivalently by:
\begin{enumerate}[(a)]
 \item {\em Eigenvalue perturbation}:
\[
 \sigma_\epsilon(M)=cl\big(\{\lambda \in \complex : \| (M(\lambda) - \lambda I) \bv\| < \eps ~\text{for some $\bv \in \complex^n$ with $\|\bv\|=1$}\}\big).
\]
 \item {\em The resolvent}:
\[
 \sigma_\epsilon(M) = cl\big(\{ \lambda \in \complex : || (M(\lambda)-\lambda I)^{-1} || > \epsilon ^{-1}\}\big).
\]
 \item {\em Perturbation of the matrix}:
\[
 \sigma_\epsilon(M) =cl\big(\{ \lambda \in \complex : \lambda \in \sigma(M(\lambda)+E) ~\text{for some}~E\in \complex^{n \times n}~\text{with}~||E||<\eps \}\big).
\]
\end{enumerate}
\end{definition}

As a consequence of definition \ref{def:psisp}, the eigenvalues of a matrix $M\in\mathbb{W}_{\pi}^{n\times n}$ belong to all its pseudospectra:  
$$\sigma(M)\subset\sigma_\epsilon(M) \ \text{for each} \ \epsilon>0.$$
The proof that definitions \ref{def:psisp}(a)--(c) are equivalent (provided the vector norm in (a) and the operator norm in (b)--(c) are consistent) relies on the proof that definitions \ref{def:psisp}(a)--(c) are equivalent for scalar valued matrices. For completeness, the proofs are included in \ref{app:psi}.

We now compare the pseudospectra of a matrix and its reduction.

\begin{example}\label{ex:ps1}
Consider the matrices $M$ and $R(M;\mathcal{B})$ given in example~\ref{ex:2} where $\mathcal{B}=\{1,2\}$. The pseudospectra of both matrices are displayed in figure~\ref{fig:ps1} for $\epsilon=1$, $10^{-1/2}$, $10^{-1}$ using the matrix 2-norm. Notice that although $0,1\in\sigma(M)$ these values do not belong to $\sigma( R(M;\mathcal{B}))$ because of cancellations resulting from the matrix reduction, i.e. $M_{\mathcal{II}}=\{0,0,1,1\}$. However, for the $\epsilon$ we consider  $0,1\in\sigma_{\eps}(R(M;\mathcal{B}))$ meaning that these eigenvalues remain as pseudoeigenvalues of the reduced matrix.

\begin{figure}
 \begin{center}
 \begin{tabular}{cc}
  \includegraphics[width=0.49\textwidth]{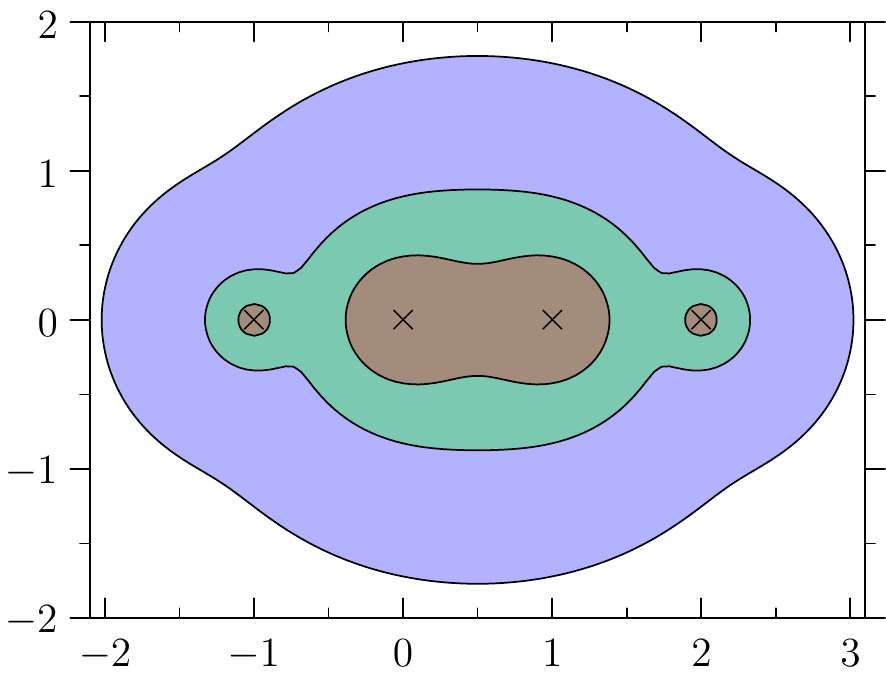} &
  \includegraphics[width=0.49\textwidth]{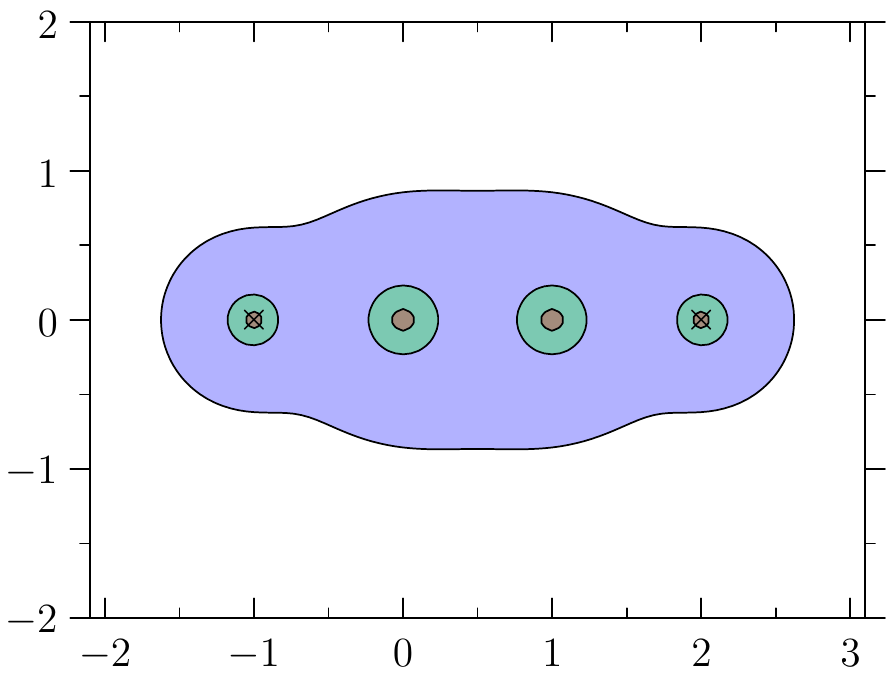}\\
   $\sigma_{\epsilon}(M)$ & $\sigma_{\epsilon}(R(M;\{1,2\}))$
  \end{tabular}
 \end{center}
\caption{Pseudospectra of the matrices given in example~\ref{ex:2} for $\eps =1$ (blue), $\eps=10^{-1/2}$ (green) and $\eps=10^{-1}$ (red), obtained with the matrix 2-norm. The respective spectra are indicated by ``$\times$''.}\label{fig:ps1}
\end{figure}
\end{example}

To give a possible physical interpretation of pseudospectra we again consider a mass-spring network.

\begin{example}\label{ex:spring1}
For the mass-spring network considered in example \ref{ex:spring0} recall that the eigenvalues of $K$ correspond to frequencies for which there exists a non-zero displacement that generates no forces on these nodes. The pseudoeigenvalues of this system have a similar physical interpretation. Namely, the pseudospectra indicate the frequencies for which there is a displacement that generates ``small'' forces relative to the (norm of the) displacement.

For example, as the frequency $\omega^2=2.1$ in figure~\ref{fig:spring2}(right) is within the green tolerance region there is a non-zero vector of displacements such that the forces generated from this displacement have norm $\eps=10^{-1/2}$ times less than the norm of this displacement vector. That is, if we only have access to the boundary nodes $\cB = \{1,4\}$ then the pseudoeigenvalues of $R_{\omega^2} ( K; \cB)$ correspond to frequencies for which there is a displacement at the boundary nodes $\cB$ that generates very small forces on these nodes. The pseudospectra regions of $R_{\lambda} ( K; \cB)$ are shown in figure~\ref{fig:spring2}(b) for $\eps=1$, $10^{-1/2}$, $10^{-1}$.

Observe that the pseudospectra of $R_{\lambda} ( K; \cB)$ are included in the pseudospectra of $K$ for a given tolerance $\eps$. That is, less access to network nodes means there are less frequencies for which displacements generate relatively small forces. Phrased less formally, the more a network is reduced, the less susceptible to perturbations its eigenvalues are.
\end{example}

\begin{figure}
 \begin{center}
 \begin{tabular}{cc}
  \includegraphics[width=0.49\textwidth]{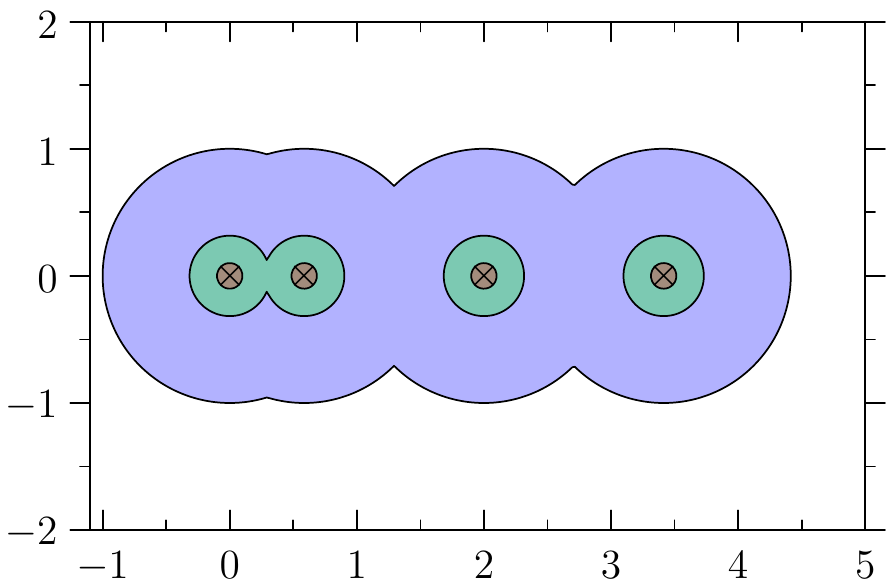} &
  \includegraphics[width=0.49\textwidth]{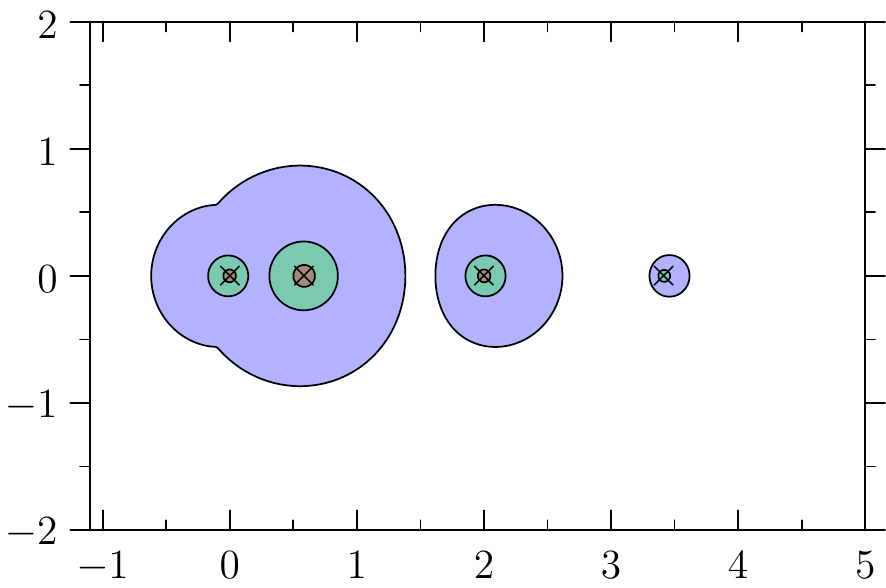}\\
$\sigma_{\epsilon}(K)$ & $\sigma_{\epsilon}(R(K;\{1,2\}))$
 \end{tabular}
 \end{center}
\caption{Pseudospectra of the stiffness matrix $K$ for the mass spring system in example~\ref{ex:spring0} and of its reduction $R_\lambda(K,\{1,4\})$. The latter corresponds to the effective stiffness of the mass-spring system when we only have access to nodes $\{1,4\}$. The tolerances shown are $\eps=1$ (blue), $\eps=10^{-1/2}$ (green) and $\eps=10^{-1}$ (red), using the matrix 2-norm. The ``$\times$'' correspond to spectra of the respective matrices.}\label{fig:spring2}
\end{figure}

Note that in both examples \ref{ex:ps1} and \ref{ex:spring1} we have $\sigma(M)\subset\sigma_{\eps}(R(M;\mathcal{B}))$ for the $\epsilon$ we consider. It seems that even under reduction, the $\epsilon$-pseudospectrum remembers where the eigenvalues of the original matrix are. However, this is not always the case, as the following example shows.

\begin{example}\label{ex:non}
Consider the matrix $M\in\mathbb{C}^{3\times 3}$ given by
$$M=\left[\begin{array}{ccc}
0&1&0\\
1&0&0\\
0&1&0
\end{array}\right],$$
with $\sigma(M)=\{0,\pm 1\}$. By reducing $M$ over $\mathcal{B}=\{1\}$ we obtain the matrix $R(M;\mathcal{B})=[1/\lambda]$ for which
$$\|(R(M;\mathcal{B})-\lambda I)^{-1}\|=\Big|\frac{\lambda}{1-\lambda^2}\Big|.$$
Hence, $0\notin\sigma_{\epsilon}(R(M;\mathcal{B}))$ for any $\epsilon$. Moreover, as $\sigma(M_{\mathcal{II}})=\{0,0\}$ for $\mathcal{I}=\{2,3\}$ it is not always the case that either $\sigma(M)$ or $\sigma(M_{\mathcal{II}})$ is contained in $\sigma_{\epsilon}(R(M;\mathcal{B}))$.
\end{example}

\subsection{Pseudoresonances} \label{sec:psire}

Recall that the resonances of a matrix $M(\lambda)\in\mathbb{W}^{n\times n}_{\pi}$ are the eigenvalues of its spectral inverse. Thus we may think of ``almost resonances'' or pseudoresonances of $M(\lambda)$ as pseudoeigenvalues of $S(M)$. The precise definition is below, together with other equivalent definitions. These are analogous to the pseudospectra definitions~\ref{def:psisp}(a)--(c).

\begin{definition}\label{def:psir}
Let $\epsilon>0$. The set of $\epsilon-$pseudoresonances of a matrix $M(\lambda) \in \BW^{n \times n}_\pi$ is defined equivalently by:
\begin{enumerate}[(a)]
 \item {\em Resonance perturbation:}
\[
 \sigma_{\eps}^{-1}(M) = cl\big(\{ \lambda \in \complex : \| (M(\lambda) - \lambda I)^{-1} \bv \| < \epsilon ~\text{for some}~ \bv\in\complex^n ~\text{with}~ \|\bv\| = 1\}\big).
\]
 \item {\em The inverse resolvent:}
 \[
 \sigma_{\eps}^{-1}(M) = cl\big(\{ \lambda \in \complex : || M(\lambda) - \lambda I || > \epsilon^{-1} \}\big).
 \]
 \item {\em Perturbation of the spectral inverse:}
 \[
 \sigma_{\eps}^{-1}(M) = cl\big(\{ \lambda \in \complex : \lambda \in \sigma(S(M) + E) ~\text{for some}~E \in \complex^{n\times n}~\text{with}~||E||<\eps \}\big).
 \]
\end{enumerate}
\end{definition}

Note that definition \ref{def:psir} is simply definition \ref{def:psisp} in which $M(\lambda)$ is replaced by the matrix $S(M)$ on the right hand side of parts (a)--(c). Hence, the equivalence of definitions \ref{def:psir}(a)--(c) follow from arguments similar those in \ref{app:psi}. Moreover,
if $M(\lambda)\in\mathbb{W}^{n\times n}_{\pi}$ then
$$\sigma^{-1}(M)\subset \sigma^{-1}_{\epsilon}(M) \ \text{for} \ \text{each} \ \epsilon>0.$$

Observe that if $w(\lambda)=p(\lambda)/q(\lambda)\in\mathbb{W}_{\pi}$ then by definition $\pi(p)\leq\pi(q)$. Hence we have the limit,
$$\lim_{|\lambda|\rightarrow\infty}|w(\lambda)|=c,$$ for some constant $c\geq 0$. Therefore $||M(\lambda)-\lambda I||=\mathcal{O}(\lambda)$ for large $\lambda$, for matrices $M \in \BW_\pi^{n\times n}$. This leads to the following remark.

\begin{remark}
If $M\in\mathbb{W}_{\pi}^{n\times n}$ then the value $\lambda = \infty$ is always a pseudoresonance. This means that for each $\eps>0$ the set $\sigma^{-1}_\eps(M)$ contains the complement of a ball centered at the origin with sufficiently large radius. (See figure \ref{fig:pr1} for example.)
\end{remark}

\begin{example}\label{ex:non2}
In figure~\ref{fig:pr1} we show the pseudoresonance regions of the matrix $R(M;\{1,2\})$ from example~\ref{ex:2} for $\eps=1$, $10^{-1/2}$, $10^{-1}$. As is shown in example~\ref{ex:3}, the inverse spectrum of $R(M;\{1,2\})$ is empty. However, the pseudoresonance regions reveal that the eigenvalues $\sigma(M_{\mathcal{II}})=\{0,1\}$ act as resonances. Specifically, $\sigma(M_{\mathcal{II}})\subset\sigma^{-1}_{\eps}(R(M;\{1,2\}))$.

In figure \ref{fig:ps1} (left) and figure \ref{fig:pr1} note that for the $\epsilon$ we consider
$$\sigma_{\eps}(R(M;\{1,2\}))\cap\sigma_{\eps}^{-1}(R(M;\{1,2\}))\neq\emptyset.$$
That is, values near the set $\sigma(M_{\mathcal{II}})=\{0,1\}$ are both $\eps$-pseudoeigenvalues and $\eps$-pseudoresonances of $R(M;\{1,2\})$.
\end{example}

As it turns out, the situation in example \ref{ex:non2} does not hold for every matrix reduction. Similar to example \ref{ex:non}, if
$$M=
\left[
\begin{array}{cc}
1&1\\
0&0
\end{array}
\right],$$
and we consider the sets $\mathcal{B}=\{1\}$ and $\mathcal{I}=\{2\}$, then one can show the set $\sigma(M_{\mathcal{II}})=\{0\}$ is not contained in $\sigma_\eps^{-1}(R(M;\mathcal{B}))$ for small $\epsilon>0$. That is, the eigenvalues $\sigma(M_{\mathcal{II}})$ do not always act as resonances of $R(M;\mathcal{B})$.

\begin{figure}
 \begin{center}
  \includegraphics[width=0.49\textwidth]{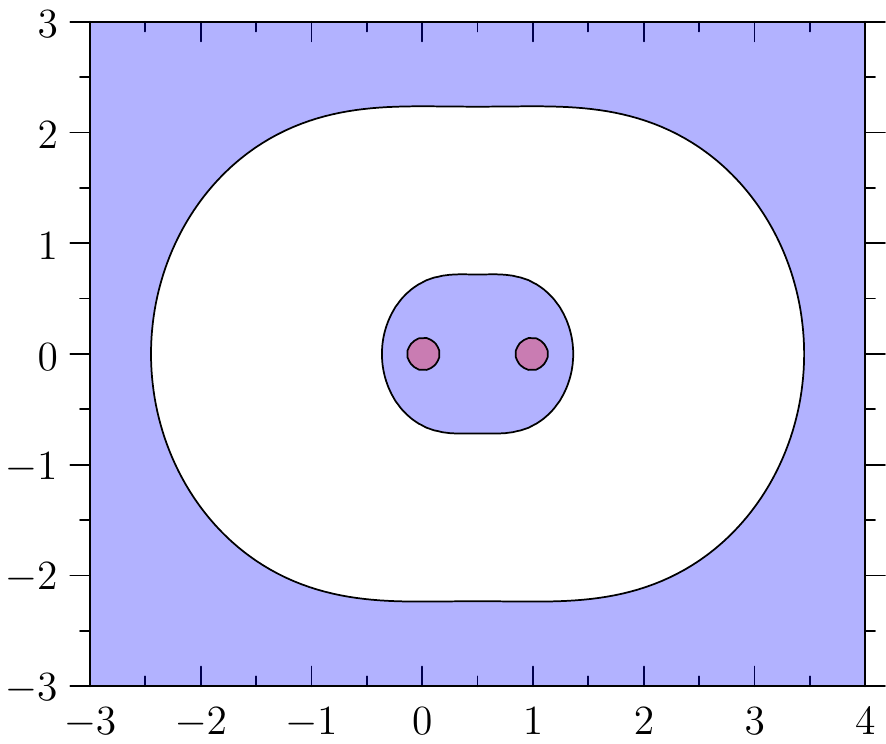}
 \end{center}
\caption{Pseudoresonance regions for the matrix $R(M;\{1,2\})$ given in example~\ref{ex:2}, with $\eps = 10^{-1/2}$ (red), and $\eps = 10^{-1}$ (blue). All the points in the display region belong to the pseudoresonance region for $\eps = 1$.}\label{fig:pr1}
\end{figure}

As with the pseudospectra studied in Section \ref{sec:psisp} we give a physical interpretation of pseudoresonances using a mass spring system.

\begin{example}\label{ex:spring2}
The mass spring system considered in example~\ref{ex:spring0} has resonances when restricted to a set of boundary nodes $\cB\subset \{1,4\}$. The pseudoresonances of the reduced system correspond to frequencies for which there is a displacement on the boundary that generates relatively large forces at these nodes. In figure~\ref{fig:spring3} we display some pseudoresonance regions of the mass-spring system restricted to the set $\mathcal{B}=\{1,4\}$.
\end{example}

As we allow $\eps$ to be any positive value there is nothing preventing an eigenvalue of a matrix $M$ from also being an $\eps$-pseudoresonance of $M$ (or a resonance from being a $\eps$-pseudoeigenvalue). In other words, we could have an $\eps>0$ for which
 \[
  \sigma^{-1}(M) \cap \sigma_\eps(M) \neq \emptyset
  ~~\text{or}~~
  \sigma(M) \cap \sigma_\eps^{-1}(M) \neq \emptyset
 \]
as the following example shows.

\begin{example}\label{ex:nondis}
Consider the following matrix $M(\lambda)\in\BW^{2 \times 2}_\pi$ given by
\[
 M(\lambda) = \begin{bmatrix} \frac{1}{\lambda-1} & 0\\ 0 & 0 \end{bmatrix}.
\]
The spectrum and inverse spectrum of $M(\lambda)$ are respectively
\[
 \sigma(M) = \{ 0 , (1\pm \sqrt{5})/{2} \}
 ~\text{and}~
 \sigma^{-1}(M) = \{1\}.
\]
Now notice that for $0 \in \sigma(M)$ we have
\[
 \| M(0) - 0 I \| = 1,
\]
which implies that $0 \in \sigma^{-1}_\eps(M)$ for all $\eps\geq1$. The resolvent of $M$ is
\[
 (M(\lambda) - \lambda I)^{-1} = \begin{bmatrix}
 \frac{\lambda -1}{-\lambda^2+\lambda+1} & 0 \\ 0 & -\frac{1}{\lambda} \end{bmatrix}.
\]
Hence, for $\lambda=1$ we have
\[
 \| (M(1) - I)^{-1} \| = 1,
\]
which means that $1 \in \sigma_\eps(A)$ for all $\eps\geq1$.
\end{example}

\begin{figure}
 \begin{center}
  \includegraphics[width=0.49\textwidth]{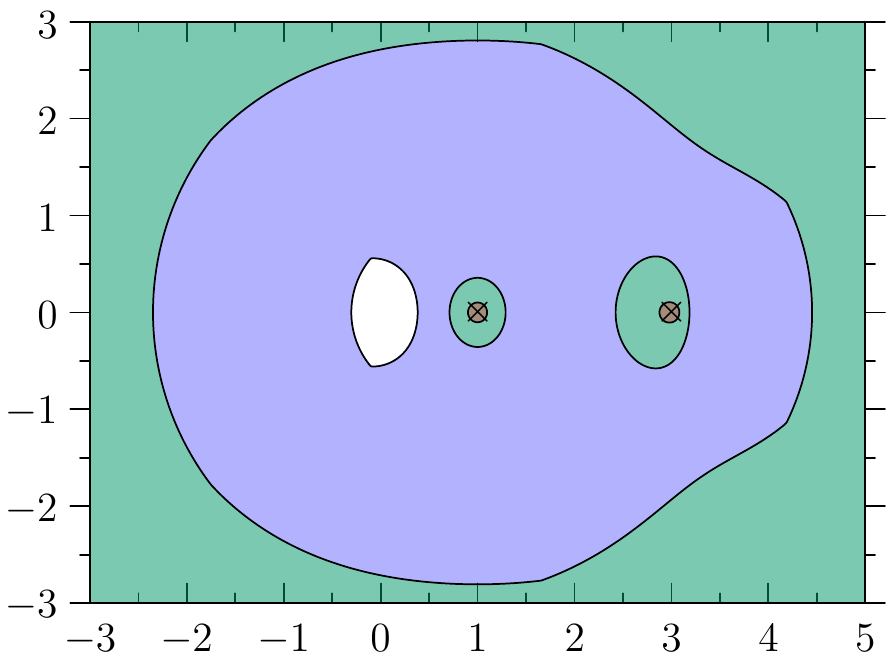}
 \end{center}
\caption{Pseudoresonance regions of the matrix $R_\lambda(K;\{1,4\})$ given in example~\ref{ex:spring0}, for $\eps = 1$ (blue), $\eps = 10^{-1/2}$ (green) and $\eps = 10^{-1}$ (red).  Resonances are shown with $\times$. All the points in the display region, the white region excepted, belong to the pseudoresonance region for $\eps=1$.}\label{fig:spring3}
\end{figure}

As the pseudoresonances of a matrix $M\in\mathbb{W}^{n\times n}_{\pi}$ can be defined in terms of the pseudoeigenvalues of the spectral inverse $S(M)$, we can generalize theorem~\ref{theorem3} as follows.

\begin{theorem}\label{thm:psir}
Suppose $M(\lambda) \in \BW^{n \times n}_\pi$ and $\epsilon>0$. Then
\begin{equation*}
 \sigma_\eps^{-1} ( M ) = \sigma_{\eps} (S(M)) \ \ \text{and} \ \ \sigma_{\eps} (M) = \sigma_\eps^{-1} (S (M) ).
\end{equation*}
\end{theorem}

\begin{proof}
Let $M(\lambda) \in \BW^{n \times n}_\pi$ and $\epsilon>0$. Observe that,
\begin{align*}
 \sigma_{\eps}(M) &= cl\big(\{ \lambda \in \complex : || (M(\lambda) - \lambda I)^{-1} || > \epsilon^{-1} \}\big); \ \text{and}\\
 \sigma_{\eps}^{-1}(S(M)) &= cl\big(\{ \lambda \in \complex : || S(M) - \lambda I || > \epsilon^{-1} \})
\end{align*}
from definitions \ref{def:psisp}(b) and \ref{def:psir}(b) respectively. Since $S(M)-\lambda I=(M(\lambda) - \lambda I)^{-1}$ then $\sigma_{\eps} (M) = \sigma_\eps^{-1} (S (M) )$. The equality $\sigma_\eps^{-1} ( M ) = \sigma_{\eps} (S(M))$ follows similarly.
\end{proof}

Because of the seemingly invertible relationship between pseudospectra and inverse pseudospectra in theorem \ref{thm:psir}, it is tempting to think the $\eps-$pseudoresonances of a matrix is the complement of its $\eps^{-1}-$pseudoeigenvalues. In general, however, the two are not equal as can be seen in the next proposition.

\begin{theorem}
For $M(\lambda) \in \BW^{n \times n}_\pi$ let $\eps>0$. Then $cl\big(\overline{ \sigma_{1/\eps} ( M )}\big) \subseteq \sigma_\eps^{-1} (M)$. However, the reverse inclusion does not hold in general.
\end{theorem}
This theorem means that, in general, there is not enough information in the pseudospectra of a matrix to reconstruct its pseudoresonances.  We now proceed with the proof of the proposition.
\begin{proof}
For $M(\lambda) \in \BW^{n \times n}_\pi$ and a matrix norm $||\cdot||$, the inequality
\begin{equation}\label{eq:lst}
||M(\lambda)-\lambda I||^{-1}\leq ||(M(\lambda)-\lambda I)^{-1}||
\end{equation}
holds for any $\lambda\in \dom(M)-\sigma(M)$. Let $\interior(\Omega)$ denote the \emph{interior} of the set $\Omega\subseteq\mathbb{C}$, i.e. the largest open subset of $\Omega$. For $\epsilon>0$, using definition \ref{def:psisp}(b)
\begin{align*}
 cl\big(\overline{ \sigma_{1/\eps} ( M )}\big)&=cl\big(\overline{cl(\{\lambda\in\mathbb{C}:||(M(\lambda)-\lambda I)^{-1}||> \epsilon\})}\big)\\
 &=cl\big(\interior(\{\lambda\in\mathbb{C}:||(M(\lambda)-\lambda I)^{-1}||\leq \epsilon\})\big)\\
 &=cl\big(\{\lambda\in\mathbb{C}:||(M(\lambda)-\lambda I)^{-1}||\leq \epsilon\}\big)
\end{align*}
Similarly, it follows from definition \ref{def:psir}(b) that
\begin{align*}
\sigma_\eps^{-1} (M)&=cl\big(\{\lambda\in\mathbb{C}:||M(\lambda)-\lambda I||>\epsilon^{-1}\}\big)\\
&=cl\big(\{\lambda\in\mathbb{C}:||M(\lambda)-\lambda I||^{-1}\leq \epsilon\}\big).
\end{align*}
By inequality \eqref{eq:lst} the set
$$\{\lambda\in\mathbb{C}:||(M(\lambda)-\lambda I)||^{-1}\leq \epsilon\}\subseteq\{\lambda\in\mathbb{C}:||(M(\lambda)-\lambda I)^{-1}||\leq \epsilon\}$$
implying the first half of the result.

To show that the reverse inclusion does not hold in general, take for instance the matrix $M(\lambda)$ from example~\ref{ex:nondis}. It is easy to compute $\| M(2) - 2I \| = 2$ and $\| (M(2) - 2I)^{-1} \| = 1$. Taking $\eps = 2/3$, we clearly have $2 \in \sigma_{2/3}^{-1}(M) \cap \sigma_{3/2}(M)$.
\end{proof}

\subsection{Pseudospectra Under Isospectral Reduction} \label{sec:redpsi}
One of the major goals of this paper is to understand how the pseudospectra of a matrix $M\in\BW^{n \times n}_\pi$ is affected by an isospectral reduction. In order to study this change in pseudospectra, we need to consider two vector norms. Specifically, we need one norm $\|\cdot\|$ defined on $\complex^n$ for the pseudospectrum of $M$ and another norm $\|\cdot\|'$ defined on $\complex^m$ ($m<n$) for the pseudospectrum of $R(M;\mathcal{B})$. Our comparison of the pseudospectra of the original and reduced matrices assumes that for $\bv = (\bv_\cB^T, \bv_\cI^T)^T \in \complex^n$ these two norms are related by
\begin{align} \label{eq:normprop}
       \norm{\bv} =
       \norm{ \begin{bmatrix} \bv_\cB \\ \bv_\cI \end{bmatrix} }
 \geq  \norm{ \begin{bmatrix} \bv_\cB \\ 0 \end{bmatrix} }
    =  \norm{\bv_\cB}'.
\end{align}
Examples of norms satisfying property \eqref{eq:normprop} are the $p-$norms for $1\leq p \leq \infty$. For the sake of simplicity, we use the same notation for both of these $\complex^n$ and $\complex^{m}$ norms.

The following theorem describes how the $\epsilon$-pseudospectrum of a matrix $M(\lambda)$ is related to the $\epsilon$-pseudospectrum of the isospectral reduction $R_\lambda(M;\cB)$. It says that the $\eps$-pseudospectra of the reduced matrix is contained in the $\eps$-pseudospectra of the original matrix for each $\eps>0$.

\begin{theorem} \label{thm:red}
For $M(\lambda) \in\BW^{n\times n}_{\pi}$ let $\mathcal{B}\subset N$. Then $\sigma_\eps(R(M;\cB)) \subseteq \sigma_\eps(M)$ for any $\epsilon>0$
provided the $\complex^n$ and $\complex^{|\cB|}$ norms in the pseudospectra definitions satisfy \eqref{eq:normprop}.
\end{theorem}

\begin{proof}
For $M(\lambda)\in\mathbb{W}^{n\times n}_{\pi}$ let $\mathcal{B}$ and $\mathcal{I}$ form a non-empty partition of $N$. We assume, without loss of generality, that for a vector $\bv\in\complex^n$ we have $\bv = (\bv_\cB^T,\bv_\cI^T)^T$.

For $\widetilde{\lambda}_0 \in\mathbb{C}$ and $\epsilon>0$ suppose there is a unit vector $\mathbf{v}_{\mathcal{B}}\in\mathbb{C}^{|\mathcal{B}|}$ such that
\begin{equation}\label{ineq:last}
||(R(M;\mathcal{B})-\widetilde{\lambda}_0 I)\mathbf{v}_{\mathcal{B}}||<\epsilon.
\end{equation}
As $\sigma(M_{\mathcal{II}})$ and $\overline{\dom(M)}$ are finite sets, then by continuity there is a neighborhood $U$ of $\widetilde{\lambda}_0$ such that
\begin{enumerate}[(i)]
\item $M(\lambda)\in\mathbb{C}^{n\times n}$ for $\lambda\in U-\{\widetilde{\lambda}_0\}$;
\item $\sigma(M_{\mathcal{II}})\cap(U-\{\widetilde{\lambda}_0\})=\emptyset$; and
\item $||(R(M;\mathcal{B})-\lambda I)\mathbf{v}_{\mathcal{B}}||<\epsilon$ for $\lambda\in U-\{\widetilde{\lambda}_0\}$.
\end{enumerate}
Observe that, for each $\lambda_0\in U-\{\widetilde{\lambda}_0\}$ it follows that the vector
$$\mathbf{v}_{\mathcal{I}}=-(M(\lambda_0)_{\mathcal{II}}-\lambda_0 I)^{-1}M(\lambda_0)_{\mathcal{IB}}\mathbf{v}_{\mathcal{B}}$$
is defined. Let $\bv = (\bv_\cB^T,\bv_\cI^T)^T$ and note that
\begin{align*}
(M(\lambda_0)-\lambda_0 I)\mathbf{v}&=
\left[\begin{array}{c}
(M-\lambda I)_{\mathcal{BB}}\mathbf{v}_{\mathcal{B}}+(M-\lambda I)_{\mathcal{BI}}\mathbf{v}_{\mathcal{I}}\\
(M-\lambda I)_{\mathcal{IB}}\mathbf{v}_{\mathcal{B}}+(M-\lambda I)_{\mathcal{II}}\mathbf{v}_{\mathcal{I}}
\end{array}\right]\Big|_{\lambda=\lambda_0}\\
&=
\left[\begin{array}{c}
M_{\mathcal{BB}}\mathbf{v}_{\mathcal{B}}-M_{\mathcal{BI}}(M_{\mathcal{II}}-\lambda I)^{-1}M_{\mathcal{IB}}\mathbf{v}_{\mathcal{B}}\\
M_{\mathcal{IB}}\mathbf{v}_{\mathcal{B}}-(M_{\mathcal{II}}-\lambda I)(M_{\mathcal{II}}-\lambda I)^{-1}M_{\mathcal{IB}}\mathbf{v}_{\mathcal{B}}
\end{array}\right]\Big|_{\lambda=\lambda_0}\\
&=\left[\begin{array}{c}
(R(M;\mathcal{B})-\lambda I)\mathbf{v}_{\mathcal{B}}\\
0
\end{array}\right]\Big|_{\lambda=\lambda_0}.
\end{align*}

By the property \eqref{eq:normprop} of the norms in $\complex^n$ and $\complex^{|\cB|}$ we must have
\begin{equation}\label{eq:basicineq}
 \| (M(\lambda_0)-\lambda_0 I)\bv \| = \| (R(M(\lambda_0);\cB) - \lambda_0 I) \bv_\cB \| < \eps.
\end{equation}
As $\mathbf{v}_{\mathcal{B}}\neq\mathbf{0}$, consider the unit vector $\bu = \bv / \| \bv \| \in \complex^n$. Again by \eqref{eq:normprop} we have $\| \bv \| \geq \| \bv_\cB \| = 1$. Hence, we get the bound
\[
 \| (M(\lambda_0) - \lambda_0 I) \bu \| = \frac{ \| (M(\lambda_0) - \lambda_0 I) \bv \|} { \| \bv \| } \leq \| (M(\lambda_0)-\lambda_0 I) \bv \| < \eps,
\]
where the last inequality comes from \eqref{eq:basicineq}. This implies $\lambda_0 \in \sigma_\eps(M)$.

As this holds for any $\lambda_0\in U-\{\widetilde{\lambda}_0\}$ then $\widetilde{\lambda}_0\in cl(\sigma_\eps(M))$. Since $\sigma_\eps(M)$ is a closed set then in fact $\widetilde{\lambda}_0\in\sigma_\eps(M)$. Since $\widetilde{\lambda}_0$ is an arbitrary point in $\sigma_\eps(R(M;\mathcal{B}))$, the result follows by inequality (\ref{ineq:last}).
\end{proof}

\begin{remark}
Theorem \ref{thm:red} states that the $\eps$-pseudospectrum of a matrix becomes a subset of this region as the matrix is reduced. However, for $\eps$-pseudoresonances of a matrix there is no such inclusion result.
\end{remark}

\begin{example}\label{ex:spring3}
In the mass-spring system of example~\ref{ex:spring0}, we consider four different sets of boundary nodes $\{1,2,3,4\} \supset \{1,2,4\} \supset \{1,4\} \supset \{1\}$. Note that theorem~\ref{thm:red} implies that the corresponding pseudospectra for a given $\epsilon$ obey the same inclusions. This is shown in figure \ref{fig:spring4} for $\eps = 1$, $10^{-1/2}$, and $\eps=10^{-1}$.

In physical terms, this means that as we increase the number of internal degrees of freedom (or decrease the number of boundary nodes), it becomes harder to find frequencies for which there is a displacement that generates forces of magnitude below a certain fixed level. Hence the less boundary nodes we have, the more robust are the frequencies that generate small forces.
\end{example}

\begin{figure}
\begin{center}
\begin{tabular}{c@{}c}
 \raisebox{5em}{\rotatebox{90}{$\eps = 1$}}  & \includegraphics[width=0.8\textwidth]{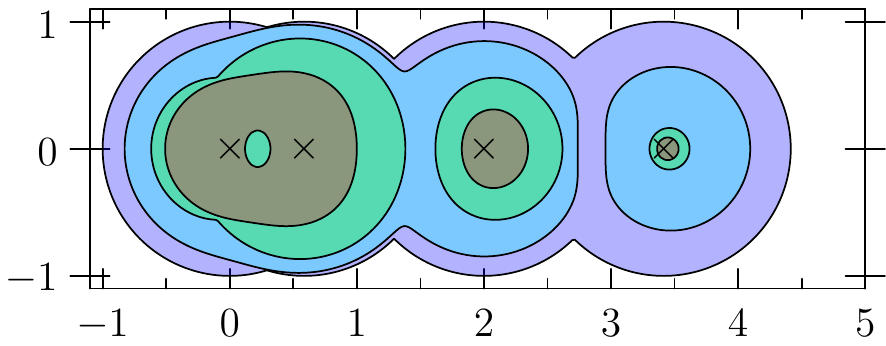}\\
 \raisebox{4em}{\rotatebox{90}{$\eps=10^{-1/2}$}} & \includegraphics[width=0.8\textwidth]{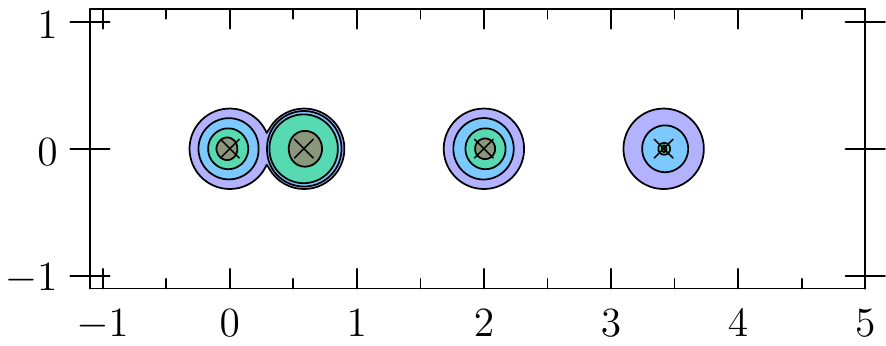}\\
 \raisebox{4em}{\rotatebox{90}{$\eps=10^{-1}$}} & \includegraphics[width=0.8\textwidth]{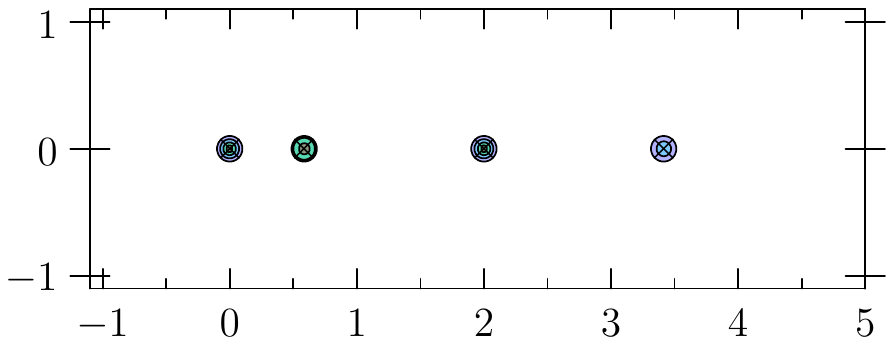}
\end{tabular}
\end{center}
\caption{Pseudospectra of the matrix $K$ from the mass-spring system of example~\ref{ex:spring0} (blue) together the pseudospectra for the reduced matrices where the terminal nodes are $\cB =\{1,2,4\}$ (cyan), $\cB = \{1,4\}$ (green) and $\cB = \{1\}$ (red). Note how the pseudospectra shrink as the number of boundary nodes decreases.}
\label{fig:spring4}
\end{figure}

Notice that the inclusion given in theorem \ref{thm:red} is not a strict inclusion. In fact, it may be the case that a matrix $M$ and its reduction $R(M;\mathcal{B})$ have the same pseudospectra as the following example demonstrates.

\begin{example}
Consider the matrix $M\in\mathbb{C}^{4\times 4}$ given by
\[
M=
\begin{bmatrix}
1&1&0&0\\
1&1&0&0\\
0&0&1&1\\
0&0&1&1
\end{bmatrix}
~~\text{and its reduction}~~
R(M;\cB) =
\begin{bmatrix}
\frac{\lambda}{\lambda-1}&0&0\\
0&1&1\\
0&1&1
\end{bmatrix},
\]
where $\cB = \{2,3,4\}$. Computing the Euclidean induced matrix norm of the resolvents we get
\[
 \begin{aligned}
 \| (M-\lambda I)^{-1} \| &= \max( |\lambda|^{-1}, |\lambda-2|^{-1} ) ~~\text{and}\\
 \| (R(M;\cB) - \lambda I)^{-1} \| & = \max ( |\lambda|^{-1}, |\lambda-2|^{-1},  |\lambda-1||\lambda|^{-1}|\lambda-2|^{-1} ).
 \end{aligned}
\]

To show that the pseudospectra of $M$ and $R(M;\mathcal{B})$ are the same, we only need to demonstrate that the norms above are equal. This happens if we can show the inequality
\begin{equation}\label{eq:ni1}
 |\lambda-1||\lambda|^{-1}|\lambda-2|^{-1} \leq \max( |\lambda|^{-1}, |\lambda-2|^{-1} ).
\end{equation}
Notice that the triangle inequality implies
\begin{equation}\label{eq:ni2}
 | \lambda - 1 | \leq  \frac{1}{2} |\lambda - 2| + \frac{1}{2} |\lambda| \leq \max(|\lambda|,|\lambda-2|).
\end{equation}
Inequality \eqref{eq:ni1} follows for $\lambda \notin \{0,2\}$ by dividing \eqref{eq:ni2} by $|\lambda||\lambda-2|$. As $\{0,2\} \subset \sigma(M),\sigma(R(M;\mathcal{B}))$, then both $0$ and $2$ are included in the pseudospectra of these matrices. We conclude that $\sigma_{\eps}( M ) = \sigma_{\eps}(R(M;\cB))$ for all $\eps>0$.
\end{example}

\section{Conclusion}
Isospectral graph reductions allow one to reduce the size of a matrix while maintaining its set of eigenvalues up to a known set. Prior to this paper it was known that a matrix could be isospectrally reduced over any principal submatrix of a particular form. One of our main results removes this restriction. This new, more general method of isospectral reduction allows one to reduce a matrix over any principal submatrix without any other consideration (other than existence). Consequently, we are able to study matrix reduction in a simpler and computationally more efficient way compared with those used in \cite{Bunimovich:2011:IGR,Bunimovich:2012:IGT,Bunimovich:2012:IC}.

An additional improvement to previous work is the introduction of a spectral inverse. The spectral inverse of a matrix, which interchanges a matrix' spectrum and inverse spectrum, allows one to use the previous results found in \cite{Bunimovich:2011:IGR,Bunimovich:2012:IGT,Bunimovich:2012:IC} to analyze the inverse spectrum of a matrix. In particular, we show that the Gershgorin-type estimates in \cite{Bunimovich:2011:IGR} can also be used to estimate a matrix' inverse spectrum.

One of our main goals here is determining whether the notion of pseudospectra can be extended to the class of matrices we consider. In fact, because a matrix with rational function entries has both a spectrum and inverse spectrum we are able to extend the notion of pseudospectrum to such matrices and also introduce the notion of inverse pseudospectrum. Moreover, we are able to show that the pseudospectrum of a matrix shrinks under reduction. Therefore, the eigenvalues of a reduced matrix are less susceptible to perturbations. This fact has implications to systems modeled by reduced matrices.

For instance, the mass spring network we consider throughout this paper is modeled using a matrix with integer entries. However, if we have access to only  some terminal nodes, the frequency response at the terminals is a matrix with rational function entries which can be obtained by reducing the stiffness matrix where all nodes are terminal nodes. Our result shows that having less terminal nodes, means the eigenvalues of the frequency response are less susceptible to perturbations than the eigenvalues of the matrix where all the nodes are terminal nodes. 

\section*{Acknowledgements}
The work of F. Guevara Vasquez was partially supported by the National Science Foundation grant DMS-0934664.

\appendix

\section{Properties of the Degree of a Rational Function}
\label{app:degree}
Suppose $w_i=p_i(\lambda)/q_i(\lambda)$ where $p_i(\lambda)$, $q_i(\lambda)\in\mathbb{C}[\lambda]$ and $q_i(\lambda)$ is nonzero for $1\leq i \leq n$. Then for $1\leq i,j\leq n$ it is easy to show the following properties hold:
\begin{align}
&\pi\big(\sum_{i=1}^n w_i\big)=\max_{1\leq i\leq n}\big\{\pi(w_i):w_i\neq 0\big\}; \label{eq:sum}\\
&\pi\big(\prod_{i=1}^n w_i\big)=
 \begin{cases}
  \dsp\sum_{i=1}^n\pi(w_i) & ~\mbox{if}~ \forall i \in \{1,\cdots,n\} ~w_i \neq 0\\
  0 & ~\mbox{otherwise};
 \end{cases}
 \label{eq:prod}\\
&\pi\big(w_i/w_j\big)=
 \begin{cases}
  \pi(w_i)-\pi(w_j)  & ~\mbox{if}~ w_i \neq 0 \\
  0 & ~\text{otherwise}
 \end{cases} \ \ \text{for} \ \ w_j\neq 0; \ \ \text{and}\label{eq:ratio}\\
&\pi(w_i-\lambda)=1 \ \ \text{for} \ \ w_i\in\BW_{\pi}.\label{eq:lambda}
\end{align}

\section{Eigenvalue Inclusions $\&$ Equivalence of Definitions \ref{def:psisp}--\ref{def:psir}}
\label{app:psi}
Here we first show that the three pseudoeigenvalue regions in definition \ref{def:psisp}(a)--(c) are equivalent and include the eigenvalues of the matrix. The proof relies on the fact that the sets
\begin{enumerate}[(a)]
 \item
$\sigma_\epsilon(M)=\{\lambda \in \complex : || (M - \lambda I) \bv|| < \eps~\text{for some}~\bv \in \complex^n~\text{with}~ ||\bv||=1\}$;
 \item
$\sigma_\epsilon(M) = \{ \lambda \in \complex : || (M-\lambda I)^{-1} || > \epsilon ^{-1}\}\cup\sigma(M)$; and
 \item
$\sigma_\epsilon(M) =\{ \lambda \in \complex : \lambda \in \sigma(M+E) ~\text{for some}~E\in \complex^{n \times n}~\text{with}~||E||<\eps \}$.
\end{enumerate}
are equivalent for any $M\in\mathbb{C}^{n\times n}$ and $\epsilon>0$. This result can be obtained by following the proof at \url{http://www.cs.ox.ac.uk/pseudospectra/thms/thm1.pdf}.

\begin{theorem}\label{thm:appB}
Let $M(\lambda)\in\mathbb{W}^{n\times n}_\pi$ and $\epsilon>0$. Then definitions \ref{def:psisp}(a)--(c) are equivalent. Moreover, $\sigma(M)\subset\sigma_{\epsilon}(M)$.
\end{theorem}

\begin{proof}
For $M(\lambda)\in\mathbb{W}^{n\times n}_\pi$ and $\epsilon>0$ let
\begin{enumerate}[(a)]
 \item
$\sigma_{\epsilon,a}(M)=\{\lambda \in \complex : || (M(\lambda) - \lambda I) \bv|| < \eps~\text{for some}~\bv \in \complex^n~\text{with}~ ||\bv||=1\}$;
 \item
$\sigma_{\epsilon,b}(M) = \{ \lambda \in \complex : || (M(\lambda)-\lambda I)^{-1} || > \epsilon ^{-1}\}\cup\sigma(M)$; and
 \item
$\sigma_{\epsilon,c}(M) =\{ \lambda \in \complex : \lambda \in \sigma(M+E) ~\text{for some}~E\in \complex^{n \times n}~\text{with}~||E||<\eps \}$.
\end{enumerate}
Suppose $\lambda_0\in\sigma(M)$. Then there is a unit vector $\mathbf{v}\in\mathbb{C}^n$ such that
$$(M(\lambda)-\lambda I)\mathbf{v}=w(\lambda)\in\mathbb{W}^n_{\pi}$$
where $w(\lambda_0)=\mathbf{0}$. Since $\sigma(M)$ and $\overline{\dom(M)}$ are finite then there is a neighborhood $U\ni\lambda_0$ such that for $\widetilde{U}=U-\{\lambda_0\}$:
\begin{enumerate}[(i)]
\item $\widetilde{U}\subset \dom(M)$;
\item $||w(\lambda)||<\epsilon$ for $\lambda\in\widetilde{U}$; and
\item $(\sigma(M)-\{\lambda_0\})\cap\widetilde{U}=\emptyset$.
\end{enumerate}
In particular, (ii) implies the set $\widetilde{U}\subset\sigma_{\epsilon,a}(M)$.

For $\lambda\in \dom(M)$ observe that the matrix $M(\lambda)\in\mathbb{C}^{n\times n}$. Since (a)--(c) are equivalent for any complex valued matrix then $$\widetilde{U}\subset \sigma_{\epsilon,a}(M)-\{\lambda_0\},\sigma_{\epsilon,b}(M)-\{\lambda_0\},\sigma_{\epsilon,c}(M)-\{\lambda_0\}.$$ This in turn implies
\begin{equation}\label{eq:last}
\sigma(M)\subset cl\big(\sigma_{\epsilon,a}(M)\big),cl\big(\sigma_{\epsilon,b}(M)-\sigma(M)\big),cl\big(\sigma_{\epsilon,c}(M)\big).
\end{equation}
In particular, if $\sigma_{\epsilon,b}(M)-\sigma(M)$ is open then $\sigma_{\epsilon,b}(M)$ is open.

Note that the norm of a vector or matrix is continuous with respect to its entries. Also, the eigenvalues of a matrix depend continuously on the matrix entries. Thus, the sets $\sigma_{\epsilon,a}(M)$, $\sigma_{\epsilon,b}(M)-\sigma(M)$, and $\sigma_{\epsilon,c}(M)$ are open.
Therefore, the set $\sigma_{\epsilon,b}(M)$ is also open.

Since the sets (a)--(c) are equivalent on $\dom(M)$ and $\overline{\dom(M)}$ is finite then
$$\sigma_{\epsilon,a}(M)\cap \dom(M)=\sigma_{\epsilon,b}(M)\cap \dom(M)=\sigma_{\epsilon,c}(M)\cap \dom(M)$$
is an open set. Taking the closure it follows that
$$cl\big(\sigma_{\epsilon,a}(M)\big)=cl\big(\sigma_{\epsilon,b}(M)-\sigma(M)\big)=cl\big(\sigma_{\epsilon,c}(M)\big)$$
implying that definitions \ref{def:psisp}(a)--(c) are equivalent. Moreover, equation \ref{eq:last} implies $\sigma(M)\subset\sigma_{\eps}(M)$.
\end{proof}

The proof that definitions \ref{def:psir}(a)--(c) are equivalent is very similar to the proof of theorem \ref{thm:appB} and is therefore omitted.

\bibliographystyle{abbrvnat}
\bibliography{psewbib}
\end{document}